\documentclass[english]{article}

\usepackage{amssymb}

\usepackage{mathrsfs}
\usepackage{amsthm}
\usepackage{amsmath}
\usepackage{amssymb}
\usepackage{esint}

 \numberwithin{equation}{section}

\newtheorem{Theorem}{Theorem}[section]

\newtheorem{thm}{Theorem}[section]

\newtheorem{prop}[Theorem]{Proposition}

\newtheorem{lem}[Theorem]{Lemma}

\newtheorem{cor}[Theorem]{Corollary}

\theoremstyle{definition}
\newtheorem{defn}[Theorem]{Definition}

\newtheorem{rem}[thm]{Remark}
  
\newcommand{\Imp}{\mbox{$\Longrightarrow$}}
\newcommand{\Iff}{\mbox{$\Longleftrightarrow$}}

\def\ie{\emph{i.e.}}
\def\Ie{\emph{I.e.}}
\def\pes{\emph{e.g.}}
\def\Pes{\emph{E.g.}}
\def\wlg{without loss of generality}
\def\A{{\mathcal{A}}}
\def\B{{\mathscr{B}}}

\def\D{{\mathcal D}}
\def\E{{\mathcal E}}
\def\F{{\mathcal F}}

\def\P{{\mathcal P}}
\def\S{{\mathscr S}}

\def\U{{\mathcal U}}

\def\QQ{{\mathcal{Q}}}
\def\Ql{{\mathcal{Q}}}

\def\FF{{\mathbb F}}
\def\EE{{\mathbb E}}
\def\N{{\mathbb N}}

\def\R{{\,\mathbb R}}
\def\Z{{\mathbb Z}}

\def\UU{{\mathbb{U}}}

\def\UU{{\mathbb{U}}}

\def\kN{\N^{*} }

\def\kZ{\Z{}^{*} }

\def\Ak{{\mathfrak A}}

\def\Nk{{\mathfrak N}}

\def\ag{{\alpha}}
\def\dg{{\delta}}

\def\bg{{\beta}}
\def\eg{{\varepsilon}}
\def\fg{{\varphi}}

\def\og{{\omega}}
\def\Og{{\Omega}}

\def\cg{{\gamma}}
\def\sg{{\sigma}}
\def\Sg{{\Sigma}}
\def\thg{{\theta}}
\def\zg{{\zeta}}

\def\ah{\aleph}

\def\max{\mbox{\rm max}\;}

\def\+#1{\vec{#1}}

\def\sp{{\mathsf {(sp)}}}
\def\fpp{{\mathsf {(fpp)}}}
\def\up{{\mathsf {(up)}}}
\def\dsp{{\mathsf {(dsp)}}}

\def\xib{{\mathbf \xi}}
\def\zgb{{\mathbf \zg}}

\def\ab{{\mathbf{a}}}

\def\xb{{\mathbf{x}}}
\def\yb{{\mathbf{y}}}
\def\zb{{\mathbf{z}}}

\def\ik{{\mathfrak{i}}}

\def\ak{{\mathfrak{a}}}

\def\sk{{\mathfrak{s}}}

\def\nk{{\mathfrak{n}}}

\def\Ik{{\mathfrak{I}}}

\def\Lk{{\mathfrak{L}}}

\def\Nk{{\mathfrak{N}}}

\def\*{\times}

\def\0{\emptyset}
\def\7{\setminus}
\def\_{\overline}

\def\<{\prec}

\def\o+{\bigoplus}

\def\incl{\subseteq}

\def\linc{\supseteq}

\def\pincl{\subset}
\def\qincl{\sqsubseteq}
\def\qinclp{\sqsubset}
\def\lincq{\sqsupseteq}

\def\la{\langle}
\def\ra{\rangle}
\def\ua{\uparrow}

\def\qed{\hfill $\Box$}

\def\RA{$\mathsf{RA} $}
\def\LA{$\mathsf{LA} $}
\def\CA{$\mathsf{CA} $}

\def\DA{$\mathsf{DA} $}
\def\PA{$\mathsf{PA} $}
\def\SA{$\mathsf{SA} $}

\def\EP{\textsf{(EP)}}
\def\HP{\textsf{(HP)}}

\def\FIP{\textsf{FIP}}

\usepackage{babel}

\begin{document}



\title{The Euclidean numbers}


\author{Vieri Benci\\
{\small Dipartimento di Matematica}\\
{\small Universit\`{a} di Pisa, Italy.}\\
{\small {\tt benci@dma.unipi.it}}
\and
Lorenzo Bresolin\\
{\small Scuola Normale Superiore, Pisa}\\
{\small {\tt lor.breso98@libero.it}}
\and
Marco Forti\\
{\small Dipartimento di Matematica}\\
{\small Universit\`{a} di Pisa, Italy.}\\
{\small {\tt marco.forti@unipi.it}}}

\date{}
\maketitle
%
%
%


\begin{abstract}
We introduce axiomatically a Nonarchimedean field $\EE$,   called the field of the Euclidean numbers, where a transfinite sum 
 indicized  by   ordinal numbers less than the first inaccessible $\Og$ is defined.  
Thanks to this sum, $\EE$ becomes a saturated  hyperreal field isomorphic to the so called Keisler field of cardinality $\Og$, and there is  a natural  isomorphic embedding into $\EE$ of the semiring $\Og$ equipped by the natural ordinal sum and product. Moreover a notion of limit is introduced so as to  obtain that transfinite sums be  limits of suitable $\Og$-sequences of their finite subsums.

Finally a notion of numerosity satisfying all Euclidean common notions is given, whose values are nonnegative nonstandard integers of ${\mathbb{E}}$. Then $\EE$  can be charachterized as the hyperreal field
generated by the real numbers  together with the semiring of numerosities
 (and
this explains the name ``Euclidean" numbers).

\end{abstract}

\begin{description}
\item[ Keywords:] 
Nonstandard Analysis,  Nonarchimedean fields,   Euclidean numerosities

\item[MSC[2010\!\!]]: 26E35, 03H05, 03C20,   03E65, 12L99
\end{description}

%

%


\section*{Introduction}

In this paper we introduce a numeric field denoted by $\mathbb{E}$, which we
name the \emph{field
of the Euclidean numbers}. The theory
of the Euclidean numbers combines the Cantorian theory of ordinal numbers
with Non Standard Analysis (NSA). 

From the algebraic point of view, the Eucliean numbers are a non-Archimedean
field with a supplementary structure (the \emph{Euclidean structure}),
which characterizes it. This Euclidean structure is introduced axiomatically
by the operation of \emph{transfinite sum}: more precisely, in 
Section \ref{field} we 
introduce sums of the type $\sum_{k}a_{k}$
where the $a_{k}s$ are real numbers, while $k$ varies  
in $\Omega$, the set of 
all ordinals smaller than the first inaccessible cardinal. 

It is worth pointing out that the inaccessibility of $\Og$ is used only in order that $\EE$ be \emph{saturated}, as real closed field; all the remaining properties of $\EE$ stated in this paper remain valid if $\Og$ is any strong limit cardinal, so that the set $\B(\Og,\R)$ of all real-valued \emph{periodic} functions on $\Og$ has size $\Og$.

We give in 
Subsection \ref{axiom}
five natural axioms that rule the behaviour of these transfinite sums. 
We list the main peculiarities of the Euclidean numbers that we
deduce from the axioms on transfinite sums:
\begin{itemize}
\item Every Euclidean number can be obtained as a transfinite sum of real
numbers;
 more generally, a transfinite sum of Euclidean numbers is well defined
in Subsection \ref{axiom}, and it can be obtained as \emph{limit of 
ordinal-indexed partial sums}, under an appropriate notion of limit,
given in Subsection \ref{count}.
\item Any accessible 
ordinal 
 $\alpha\in\Og$ can be identified with the transfinite
sum of $\alpha$ ones in $\EE$; this identification
is consistent with the so called \emph{natural ordinal operations} $+$ and $\cdot$
(see Subsection \ref{ordeucl}), so
the field of the Euclidean numbers can be considered as a sort of \emph{natural}
extension of the \emph{semiring} of the (accessible) ordinal numbers.
\item The field $\mathbb{E}$ is saturated with respect to the order 
relation, actually 
it is the \emph{unique saturated real closed field having the 
cardinality of the first inaccessible cardinal} $\Omega$ (see Subsection \ref{sat}).
This property implies that every ordered field having cardinality
less than or equal to $\Omega$ is (isomorphic to) a subfield of $\mathbb{E}$.
A model of the field $\EE$ is obtained in Section \ref{constr} as a limit 
 ultrapower of $\R$ modulo a suitable ultrafilter $\U$ on $\Og$. 
\item The Euclidean numbers are a \emph{hyperreal field}; more precisely 
$\EE$
is isomorphic to the hyperreal Keisler field introduced in \cite{keisler76};
the Keisler field is the unique saturated (in the sense of NSA) hyperreal
field having the cardinality of $\Omega$ (see Subsection \ref{Keis}).
 \item The Euclidean numbers are strictly related to the notion of 
\emph{numerosity}, introduced in \cite{benci95b,BDNlab,BDNF1} and devloped
in \cite{BDNFuniv,DNFtup,QSU,FM}, so as to save the five Euclidean common notions
(see Section \ref{num}). In fact, $\EE$ can be charachterized as
the hyperreal field generated by the real numbers and the semiring
of numerosities, provided that the numerosity is defined on \emph{a coherent
family of labelled sets containing the accessible ordinal 
numbers}.
The numerosity theory provided by the Euclidean numbers satisfies
the following properties which alltogether are not shared by other numerosity
theories (see Subsections  \ref{saveu}-\ref{eunum}):
\begin{itemize}
\item each set (in this theory) is equinumerous to a set of ordinals,
\item the set of numerosities $\mathfrak{N}$ is a positive subsemiring
of nonstandard integers, that generates the whole $\mathbb{Z}^{*}$.
\end{itemize}

\end{itemize}

We have chosen to call $\EE$ the field of the \textit{Euclidean}
numbers for two main reasons: firstly, this field arises inside of a numerosity
theory (including all bounded subsets of $\Omega$), whose main aim is to
save all the Euclidean common notions, including the fifth one:
    \emph{~``The whole is greater
    than the part",~ }
 \noindent
in contrast to the Cantorian theory of cardinal numbers. 
    
The second reason is that, in our opinion, the field $\EE$ describes the
Euclidean continuum better than the \emph{real field} $\R$, at least 
when looking for a \emph{set theoretic} interpretation
of the \emph{Euclidean geometry}. This last point has been 
dealt with in \cite{BF16,BF19} and will be shortly outlined in the Appendix.

\section{Notation and preliminary notions}
\selectlanguage{english}

Let $\Omega$ be the least (strongly) \emph{inaccessible} cardinal. Or 
better, taking into account that in what follows the ordinals 
are  viewed \emph{``\`a la Cantor''} as \emph{atomic} numbers, which 
are not 
\emph{identified with}, but rather considered as the 
\emph{order types of} the set of all the smaller ordinals, 
let $\Og$ be the 
\emph{set of all accessible ordinals}, and in general let $\Og_\ag=\{\bg\in\Og\mid \bg<\ag\}$.

\subsection{Operations on $\Omega$}\label{opord}

Since we use the ordinary symbols $\cdot$ and $+$ for the operations on 
the Euclidean numbers that we shall define in section \ref{axiom}, 
and among them we shall include the ordinals, the \emph{usual ordinal 
mutiplication  and addition} on $\Omega$ will be denoted by $\odot$ and $\oplus$, 
respectively, whereas $\cdot$ and $+$ will correspond to the so 
called \emph{natural operations}, which coincide with the operations in $\EE$, see Theorem \ref{ordeu} below. 

Given ordinals $\ag, j\in\Omega$, there exist uniquely determined 
ordinals $k\in\Omega$ ed $s<2^{j},$ such that
\[
\ag=\left(2^{j}\odot k\right)\oplus s.
\]
%

Recall that each ordinal has a unique \emph{base-$2$ normal form}
\[
\ag=\sum_{n=1}^{N}2^{j_{n}}
\]
where  $n_{1}<n_{2}\Rightarrow j
{}_{n_{1}}>j{}_{n_{2}}$.

As we identify the ordinals in $\Og$ with numbers of the field $\EE$, 
we shall simply write the normal form\
$\ag=\sum_{n=1}^{N}2^{j_{n}}$,\ 
independently of the ordering of the exponents. But one has to be 
careful: sum and product agree with the \emph{ordinary ordinal 
operations} only when \emph{the exponents are decreasing (and 
integer coefficients are put on the right side)}. On the other hand, the exponentiation between ordinals  is intended as the usual \emph{ordinal exponentiation}, and so it differs from the \emph{nonstandard extension of the real exponentiation} as defined in $\EE$.

In particular  $2^\og=\og$, and  the power $2^\ag=\og^\ag$ whenever $\ag= \og\odot\ag$. It follows that the fixed points of the function $\ag\mapsto 2^\ag$ are $\og$ and and the so called $\eg$-numbers $\eg$ such that $\og^\eg=\eg$.

\subsection{Finite sets of  ordinals}\label{fin}
The usual \emph{antilexicographic wellordering} of the finite sets of ordinals is defined by 
$$L_1 < L_2\ \  if\  and\ only\ if\ \ \max (L_1 \bigtriangleup L_2 ) \in L_2.$$

In this ordering $2^\ag$ is the order type of the set $\P_{fin}(\Og_\ag)$ of all finite sets of ordinals less than $\ag$, hence the family $\Lk=\P_{fin}(\Og)$ of all finite subsets of $\Og$ can be isomorphically indexed by $\Og$. Therefore we shall denote by $L_\ag$ the $\ag$th set of ordinals, namely
$$L_\ag =\{ \ag_1, ... ,\ag_n\}\ \  for \ \ \ag=\sum_1^n 2^{\ag_i}.$$ 
In particular 
$$L_0=\0,\ \ \ L_{2^\ag}=\{\ag\},\ \ \ and\ \ \ L_{2^\ag+\bg}=\{\ag\}\cup L_\bg\ \ for\ all\ \bg<2^\ag.$$
The order isomorphism $\alpha\mapsto L_\ag$ between $\Og$ and $\Lk$ allows to single out a restriction of the ordinal ordering on $\Og$, that will be basic in the following sections, namely
the \emph{formal inclusion}  $\sqsubseteq$, that
corresponds to ordinary inclusion:
\begin{defn}
%
%
   \textbf{(formal inclusion)}:${~}$\\
    Given $\ag,\bg\in\Og$ we say that $\ag$ is \emph{formally included} in $\bg$ (written $\ag\qincl\bg$) if and only if $L_\ag\incl L_\bg$.
  
  Hence\footnote{
~The name formal inclusion should  also recall that the respective 
base-$2$ normal forms are indeed 
contained one inside of the other one.}
 ${~~} \ \ \ \ 
\ag=\sum_{i\in I}2^{{i}}\ \sqsubseteq\ 
\bg=\sum_{h\in H}2^{{h}}
\Longleftrightarrow\ \, I\incl H.\ 
$
\end{defn}
\medskip
So 
the formal inclusion $\bg\qinclp\ag$ reflects the ordinary inclusion between the corresponding finite sets $L_\bg\pincl L_\ag.$
In particular  the following useful properties hold:
\begin{itemize}
  \item $\,0\sqsubseteq \ag$
for all $\ag\in\Omega$. 
  \item  $|\{\bg\mid \bg\qincl \ag\}|=2^{|L_\ag|}$ 
is \emph{finite} for all 
    $\ag\in\Omega$;
  \item 
The structure $(\Og,\sqsubseteq)$ is a
a  \emph{lattice} isomorphic to $(\Lk,\incl)$, where \emph{supremum} $\ag\vee\bg$ and \emph{infimum} $\ag\wedge\bg$ are defined by

 $~~~~~~~~~~~~~~~~~~L_{\ag\vee\bg}=L_\ag \cup L_\bg~$\ and\ 
$~L_{\ag\wedge\bg}=L_\ag \cap L_\bg$.
%
 \item for all $h,k,\ag\in\Og$ one has\ $h,k\in L_\ag\ \ \Iff\ \ 2^h\vee 2^k\qincl \ag.$


\item for $\eta>0$ and all $\ag,\bg\in\Og $  one has the following useful criterion:
\begin{description}
    \item[\textbf{(C)}] 
$\ \ k<2^{\eta}\ \ \Imp\ \ 
    (\,2^{\eta}\odot \bg+k\,\qincl 
    \ag\ \ \Iff \    \    2^{\eta}\odot \bg,\, k\ \qincl \ag\,)$ 
\end{description}
\end{itemize}

\smallskip
In order to deal with the Euclidean numbers, we single out  the following class of ordinals (or better of finite sets of ordinals)

\begin{defn}

An ordinal $\ag$ (and the corresponding set $L_\ag$)  is $(\eta,h)$-\emph{complete}  if 
 \ for\ all\ $\cg<2^\eta$\ $$\cg\qincl\ag\ \ \Iff\ \ ({2^\eta})^h\odot \cg\qincl \ag,\ \ \textrm{or\ equivalently}\ \ 
k\in L_\ag\ \ \Iff\ \ \eta\odot h\oplus k\in L_\ag$$
%
\end{defn}
%
%
 
By criterion $\textbf{(C)}$, the set of the $(\eta,h)$-complete ordinals is 
$$D(\eta,h)=\{\ag\in\Og\mid\forall\cg, k<2^\eta\,\big(\, \cg,k\qincl\ag\  \Iff\ (2^{\eta\odot h}\odot\cg)\vee k\qincl \ag\big)\}.$$
In particular, $D(\eta,h)=\Og$ when $\eta\odot h=0$, and 
$$D(\eta,1)=\{\ag\mid\forall k,\cg<2^\eta\,(2^\eta\odot \cg\oplus k\qincl\ag\ \Iff\ k,\cg\qincl\ag)\}.$$
\begin{lem}\label{compl}
Let $D(\eta,h)$
 be the set of all $(\eta,h)$-complete ordinals, and let $$C(\bg)=
\{\ag\mid \bg\qincl\ag\}= \{\ag\mid L_\bg\incl L_\ag\} $$  be the \emph{cone} over $\bg$ with respect to formal inclusion. 

Then the family
 $~\D=\{D(\eta,h)\mid \eta,h\in\Og\,\}\cup\{C(\bg)\,\mid \bg\in\Og \,\}~$
enjoys the finite intersection property $\emph{ \FIP}.$

Hence $\D$ 
generates a filter $\QQ$ over $\Og$ (and correspondingly over $\Lk$).

\end{lem}

\begin{proof}
Let $\ E,H,B\in\P_\og(\Og)$ be finite sets of ordinals, and let
  $$D(E,H,B)=\{\ag\in\Og\mid\forall\eta\in E\,\forall h\in H\,\forall\bg\in B.\ \ag\in D(\eta,h)\cap C(\bg)\}.$$

 Since one has $C(\bg)\cap C(\bg')=C(\bg\vee\bg')$,  we may assume \wlg\ that $B=\{\bg\}$, and show that  $D(E,H;\{\bg\})$ is nonempty.
 
Let $\eta_1<...<\eta_n$, and let $E_i=\{\eta_1,\ldots,\eta_i\}$. 
Put $\ag_0=\bg$, and
define  inductively the ordinals $\ag_i\in E_i$ as follows:
\begin{itemize}
\item let $L(i)=L_{\ag_{i-1}}\cup\{k<{\eta_i}\mid\exists h\in H\,({\eta_i}\odot h\oplus k\in L_{\ag_i-1})\}$, and pick the 

\noindent
ordinal $\ag_i$  such that  $~~L_{\ag_i}=L(i)\cup\{\eta_i\odot h\oplus k\mid h\in H, \eta_i> k\in L(i)\}$.

%
%
\end{itemize}

Then clearly $\ag_n$ belongs  to $D(E,H;\{\bg\})$.
\qed
\end{proof}
\medskip

\begin{rem}\label{qeta} If  $E,H,B\in\P_\og(\Og_\eg)$, with $\eg=2^\eg$, then the ordinal $\ag_n$ of the above proof can be taken smaller than $\eg$. Hence the family
 $$\D_\eg=\{D(E,H,B)\cap\Og_\eg \mid\,E,H,B\incl\Og_\eg \}$$
  enjoys the \FIP, and generates a filter $\QQ_\eg$ on $\Og_\eg$. 

\end{rem}
\section{The field of the Euclidean numbers}\label{field}
We introduce the field of the \emph{Euclidean numbers} as constituded 
of all \emph{transfinite sums} of real 
numbers, of length equal to some accessible ordinal in $\Omega$.
In order to avoid the  antinomies and paradoxes that might affect the summing up of infinitely many numbers, it seems essential to consider only sums of indexed elements. The choice of ordinal numbers as indices seems particularly appropriate, given their natural ordered structure. Moreover,
we shall ground on the lattice structure given by the \emph{formal inclusion} $\qincl$
introduced in Subsection \ref{fin}. 

\subsection{Axiomatic introduction of the Euclidean numbers as 
infinite sums}\label{axiom}

%
Let $\EE$ be an \emph{ordered superfield of the reals},\footnote{~
 For sake of clarity, we denote general
Euclidean numbers by \emph{greek} letters $\sg,\tau,\xi,\eta,\zeta$,  
and real numbers by \emph{latin} letters $w,x,y,z$. The ordinal indices are denoted indifferently either by latin letters $i,j,h,k$ or by greek letters $\ag,\bg,\cg,\dg$. The $\Og$-sequences are denoted by the corresponding \emph{boldface} letters.
} and
assume that a \emph{transfinite  sum} 
\[ \sum {\xib}=\sum_{k}\xi_{k}\] 
is defined  for all eventually zero $\Og$-sequence $\mathbf{\xi}=\la \xi_{k}\mid k\in\Og\ra$ of elements of $\EE$. (We denote by $\mathscr{S}(\Omega,X)=X^{(\Og)}$ the set of all eventually zero $\Og$-sequences from $X$.)

Remark that we intend that any transfinite sum comprehends all summands $\xi_k,\, k\in\Og$. 
When needed, we can restrict the sum to any subset $K\incl\Og$ by putting
$$\sum_{k\in K} \xi_k=\sum_{k} \zg_k,\ \ \textrm{with} \ \ \zg_k=\xi_k\chi_K(k), \ \textrm{and}\ \
\chi_K(k)=\begin{cases}
    1  & \text{if } k\in K,  \\
    0  & \text{otherwise}.
\end{cases}
$$  
We make the natural assumption that  a \emph{transfinite sum} coincides with the 
\emph{ordinary
 sum} of the field $\EE$ when the number of \emph{non-zero summands} is \emph{finite}.
 
\medskip{}
 We call $\EE$ the \emph{field of the Euclidean numbers} if the following axioms are satisfied.

\begin{description}
 \item[\LA]  \textbf{\emph{Linearity  Axiom}: }
  The transfinite sum is $\R$-linear, \ie
  $$s\sum_{h}\xi_{h}+t\sum_{h}\zeta_{h}=\sum_{h}(s\xi_{h}+t\zeta_{h})\ \ \ for\ all\  \
s,t\in\R\ \ and\ \ {\xib},{\zgb}\in\EE^{(\Og)}.$$

  \item[\RA] \textbf{\emph{Real numbers  Axiom}: }
 $$ For\ all\ \ \xi\in \EE\ \ there\ exists\   \
 \mathbf{x}\in\R^{(\Og)}\ \ such\ that \ \
 \xi= \sum_{h} x_h.~~$$
 \end{description}
 Grounding on the axiom \RA, in  the following axioms we restrict ourselves 
to considering 
transfinite sums of \emph{real numbers}.
Firstly, an axiom  for comparing transfinite sums:
\begin{description}
    \item[\CA \textbf{\emph{ Comparison Axiom}}] ${~}$
 \emph{For all
$\mathbf{x,y}\in\R^{\Og}$, 
$$\ \ \ \exists\bg \ \forall
\ag\lincq \bg \ \big(\sum_{k\qincl\ag}x_{k}\,\leq\,\sum_{k\qincl\ag}y_{k}\,\big) \ 
\, 
\ \  \Imp \ \ \
\sum_{k}x_{k}\leq\sum_{k}y_{k}\, ~~~
$$}
(Remark that the sums $\sum_{k\qincl\ag}$  are \emph{ordinary finite sums of real numbers}.)
\end{description}

\medskip
We define also a  \emph{double sum}:
$$ \sum_{h,k} x_{hk}=
\sum_{j} y_j, \ \ where\ \ \ y_j=\sum_{h\vee k=j}x_{hk}.$$

(Again, the sum  $\sum_{h\vee k=j}x_{hk}$\, is an \emph{ordinary finite sum of real numbers}.)

\medskip
Recall that $L_{h\vee k}=L_h\cup L_k$, hence
$\ \ \sum_{h,k\qincl\ag}x_{hk}=\sum_{j\qincl\ag}y_j$, and
we have the
   \textbf{\emph{Double sum comparison criterion:  }}
 
 \medskip 
   $\ \ \exists\bg \ \forall
\ag\lincq \bg \ \big(\, \sum_{h,k\qincl\ag}a_{hk}\le\sum_{h,k\qincl\ag}b_{hk}\,\big)\ \ \
 \Imp  \ \ \
     \sum_{h,k}a_{hk}\le \sum_{h,k}b_{hk}.$ 
 
     \medskip

We give an axiom that linearizes the double sum so as to be  consistent with the embedding into $\EE$ of the ordinals  as transfinite sums of ones (see next subsection).
\begin{description}
\item[\DA]  \textbf{\emph{Double sum axiom}: } If $x_{hk}=0$ for $h,k\ge\eta$, then
    $$ \sum_{h,k}x_{hk}=
\sum_{i}y_{i}\ \ {where}\ \ 
 y_{i}=\begin{cases}
\,x_{hk}\!\!&if\,\,\,\,i=2^\eta\odot h + k\,\\
0\!\!&otherwise
\end{cases}.$$
\end{description}


The double sum allows to compute the products according to the following axiom:
\begin{description}
\item[\PA]  \emph{\textbf{Product axiom}}:  
    $$(\sum_{h}x_{h})(\sum_{k}y_{k})=\sum_{h,k} x_{h}y_{k}.$$
\end{description}

\medskip
 
In general, the double sum is different from the corresponding sum of sums, as we shall see below. So we give an axiom in order to simplify a sum of sums:
\begin{description}
\item[\SA]  \textbf{\emph{Sum axiom}: } If $x_{hk}=0$ for $h,k\ge\eg=2^\eg$, then
    $$ \sum_{h} \sum_{k}x_{hk}=
\sum_{i}y_{i}\ \ {where}\ \ 
 y_{i}=\begin{cases}
 \,x_{hk}\!\!&if\,\,\,\,i=\eg^h\odot k,\  k\ne 0, h\,\\
 x_{hh}+x_{h0}\!\!&if\,\,\, i=\eg^h\odot h, \ h\ne 0\\
\,x_{00}\!\!&if\,\,\,\,i=0\,\\
0\!\!&otherwise
\end{cases}.$$
\end{description}

\medskip

\begin{rem}
The use of the ordinal $\eg^h\odot k$ as code of the pair $(h,k)$ in the axiom \SA, instead of the simpler $2^\eta\odot h+k$ used in the axiom \DA, is due to the fact that a double sum is necessarily \emph{symmetric}, whereas, as we shall see below, the \emph{sums cannot be interchanged }when summing arbitrary Euclidean numbers. This choice has a drawback in the fact that all pairs $(h,0)$ would  receive the code $0$, so they have to be dealt with separately, as we did in the axiom. 


\noindent
\emph{\textbf{CAVEAT}}: It is not true, in general, that one can change the order of summation, as in the case of a double sum.

\Pes\ let $\ x_{hk}\in\{0,1\},\  h,k=0$ for $  h,k\ge \og$ be chosen so as to have exactly one $1$ in every horizontal line $k= cost$,  exactly two $1$s in every vertical line $h= cost$  (and one can arrange so as to obtain various different values of the number $s_j$ of the ones on the border $h\vee k=j$ of the square $\{h\incl j\}\*\{k\incl j\}$). Then
$$\xi_k=\sum_{h} x_{hk}=1\ \ for\ 0\le k<\og, \ \ \ \zg_h=\sum_{k} x_{hk}=2\ \  for\ 0\le h<\og,\ \
$$ hence
$$ \sum_h2\chi_\og(h)=\sum_h\zg_h=\sum_{h}\sum_{k}x_{hk}=\sum_k2\xi_k=2\cdot\sum_{k}\sum_{h} x_{hk},$$ 
while   \   $\sum_{h,k} x_{hk}=\sum_j s_j\chi_\og(j)$.
\end{rem} 
\bigskip

Surprisingly enough, these simple and natural axioms are all that is 
needed in order to endow $\EE$ with a very 
rich structure, as we shall see in the sequel. 

We begin with a few simple consequences.

\medskip

\begin{itemize}
%
%
           
\item  \textbf{\emph{Translation invariance }}\\
          If $x_k=0$ for $k\ge\eta$, then, for all $h\in\Og$,
$$\sum_{k}x_{k}=\sum_{i}y_{i},\ 
    \ \mathrm{where}\ \ y_{i}=\begin{cases}
\,x_{k}\!\!&if\,\,\,\,i=2^\eta\odot h+k\\
0\!\!&otherwise
\end{cases}.$$

In fact, put $x_{jk}=\begin{cases}
\,x_{k}\!\!&if\,\,\,\,j=0,h\\
0\!\!&otherwise
\end{cases}
. $

Then $2x_{k}=x_{k}+y_{2^\eta\odot h+k}$ and so,
\medskip
by the axioms \LA\ and  \DA,
$$ 
2\sum_k x_k=\sum_k x_{0k}+\sum_k y_{2^\eta\odot h +k}=\sum_{j,k} x_{jk}=\sum_{\ell}\sum_{j\vee k=\ell}x_{jk}=$$
$$=\sum_{\ell}x_{0\ell}+\sum_{\ell} \sum_{h\vee k=\ell}x_{hk}=
\sum_{k}x_{k}+\sum_{i}y_{i}$$.

\item \textbf{\emph{Finite sums}}\\
    The 
initial assumption that \emph{sums of finitely many non-zero 
elements receive their natural values in} 
$\EE$ can be deduced from 
the axioms \LA, \RA, \PA, \SA, and \CA.

Given $\xi_h\ne 0$ only  for $h= h_1,\ldots,h_n$, assume \wlg\ that 
$\xi_h=\sum_k x_{hk}$ with $x_{hk}=0$ for $k\ge\eg=2^\eg$, and for all $k$ if
 $h\ne h_1,\ldots,h_n$.

\smallskip \noindent
Put
 $z_k=\sum_{j=1}^n x_{h_j k}$ and let $y_i$ be as in the axiom \SA. Then

$\ \sum_h\xi_h=\sum_iy_i,$\ \ by\ \SA\,\ and\
\ $\xi_{h_1}+\ldots+\xi_{h_n}=\sum_kz_k,$\  by\ \LA.

\smallskip\noindent
Pick $\ag\lincq\bigvee_{j=1}^n(h_j\vee\eg^{h_j})$, so \ $h_j\qincl\ag,$ and $~ \eg^{h_j}\odot k\qincl\ag\ \Iff\ k\qincl\ag$. Then
$$\sum_{k\qincl\ag} z_k=\sum_{k\qincl\ag}\sum_{j=1}^n x_{h_jk}=
\sum_{ k\qincl\ag}\sum_{j=1}^n\big( y_{\eg^{h_j}\odot k}\big)=
\sum_{i\qincl\ag}y_i.$$
Hence $~\sum_kz_k=\sum_iy_i.$

\end{itemize}
     \begin{rem}\label{filt}
     The comparison axiom \CA\ could be strengthened by 
	 replacing the \emph{cone filter }generated by the {cones} 
	 $C(\ag)=\{\bg\in\Og\mid \ag\qincl\bg\}$ by a suitable finer
	 filter, so as to obtain further properties.
	 In fact, in order to model the axioms \DA\ and \SA, we should use a much 
	 finer filter, that contains the filter $\Ql$ generated by the
	 \emph{complete ordinals} of subsection \ref{fin}.
\end{rem}

\subsection{Ordinal numbers as Euclidean numbers}\label{ordeucl}
An important consequence of the axioms is the existence of a natural isomorphic 
embedding of $\Og$ (as \emph{ordered semiring} with \emph{natural} sum and product) 
into $\EE$:

\begin{thm}\label{ordeu} ${~}$
 Define
$
\Psi:\Omega\longrightarrow\mathbb{E}\
$
 by\
$\
\Psi(\alpha)=\sum_{k}\chi_{\Og_\ag}(k)
$,\ where $ \chi_{\Og_\ag}$ is the characteristic function of the set $\Og_\ag=\{\bg\in\Og\mid \bg<\ag\}$.
Then, for all $\ag,\bg\in\Og$:
\begin{itemize}
  \item[$(i)$]  $~~\ag<\bg\ \ \Iff\ \ \Psi(\alpha)<\Psi(\bg)$;
  \item[$(ii)$] $~\Psi(\alpha+\beta)=\Psi(\alpha)+\Psi(\beta)\ \ \mathit{and}\ \
\Psi(\alpha\cdot\beta)=\Psi(\alpha)\cdot\Psi(\beta).$
\end{itemize}

\end{thm}

\begin{proof}

\noindent
$(i)$. The summmands in $\Psi(\alpha)$ are an initial segment of those in $\Psi(\bg)$, and all nonzero summands are positive, so $(i)$ is immediate.

\medskip
\noindent
$(ii)$
In order to prove that $\Psi(\alpha+\beta)=\Psi(\alpha)+\Psi(\beta)$,
it suffices to show by induction on $n$  that $\Psi$ preserves the (decreasing) 
base-$2$ normal form, \ie
$$\ag=\sum_{i=1}^{n}2^{j_{i}}\ (j_{1}> j_{2} > \ldots > j_{n})\ \
\Imp\ \ 
\Psi(\ag)=\sum_{i=1}^{n}\Psi(2^{j_{i}});$$
now 
$~~\sum_k\chi_{[0,\cg)}(k)=\sum_k\chi_{[2^\eta\dot h,2^\eta\cdot h+\cg)}(k)~$ whenever $\cg<2^\eta$,
  by translation invariance, hence one has, for $\ag=\sum_{i=1}^n 2^{j_i}$ and $\bg=2^{j_{n+1}}$
$$\Psi(\sum_{i=1}^n 2^{j_i})+\Psi(2^{j_{n+1}})=
\sum_k\chi_{[0,\ag)}(k) +\sum_k\chi_{[0,\bg)}(k)= $$
$$=\sum_k\chi_{[0,\ag)}(k) +\sum_k\chi_{[\ag,\ag+\bg)}(k)=
\sum_k\chi_{[0,\ag+\bg)}(k)=
\Psi(\sum_{i=1}^{n+1}2^{j_i}).$$
\smallskip

Remark that the equality holds also when $j=h$, thus giving 
$$2\Psi(2^{j})=\Psi(2^{j})+\Psi(2^{j})=
\sum_k\chi_{[0,2^j+2^j)}(k)=
\Psi(2^{j}+2^{j})=\Psi(2^{j+1}).$$

Finally,  the natural operations on ordinals are commutative, associative, and distributive, hence the multiplicative property 
$\Psi(\alpha\cdot\beta)=\Psi(\alpha)\cdot\Psi(\beta)$ needs to be proved 
only for ordinals of the form $2^{\ag}=2^{\sum_{i=1}^{n}2^{j_{i}}}
=\prod_{i=1}^{n}2^{2^{j_{i}}}$.

Put, as above, $\ \ag=\sum_{i=1}^n 2^{j_i}$,  $\bg=2^{j_{n+1}}$, and
$y_{2^\ag\odot h+k}=\chi_{[0,2^\bg)}(h)\chi_{[0,2^\ag)}(k)$;  then 
 $$
\Psi(2^{\ag+\bg})=
\sum_{ j}\chi_{[0,2^{\ag+\bg})}(j)=
\sum_jy_j=
\sum_{h,k} \chi_{[0,2^\bg)}(h)\chi_{[0, 2^\ag)}(k)=
\Psi(2^{\ag})\Psi(2^{\bg})$$
\noindent
%
\qed

\end{proof}
\bigskip
By virtue of this theorem, we may, as stipulated at the beginning, identify each ordinal $\ag\in\Omega$
with the corresponding Euclidean number $\Psi(\ag)$, so as 
to obtain that\ 
$\,\Og\incl\EE$, exactly as we have assumed $\R\incl\EE$.
Since we prefer to have the field $\EE$ as a \emph{set of atoms}, 
 this is the reason why we  have viewed from the beginning each ordinal number $\ag\in\Og$ \emph{``\`a 
la Cantor''}  as \emph{the order-type 
of}, and  not 
 \emph{identified  {``\`a la Von Neumann''} with}  the 
corresponding initial 
segment ${~}\Og_{\ag}=\{\bg\in\Og\incl\EE\,\mid\, \bg<\ag\,\}.$ 

\begin{rem} The meaning of the \emph{natural product} between ordinal  numbers is \emph{not easily understood}, when
defined through the order type of the appropriate well ordering of the cartesian product, which is quite
different from the usual well ordering of $\Og\*\Og$. 
(In fact, it is usually defined through the Cantor normal form.)
On the contrary,  thinking of an ordinal number as a \emph{Euclidean number},
namely as a \emph{transfinite sum of ones}, makes appear quite natural 
the meaning of the product, as given by the the product formula. 
Also the ordering of the ordinals is clearly the one induced by $\EE$, 
 because ordinals are transfinite sums of ones 
\emph{without zeroes in between}.
\end{rem}


%

%

\subsection{The counting functions}\label{count}
Recall that\
$\
\mathscr{S}(\Omega,\EE)=\EE^{(\Og)}=\left\{ \xib\in\mathbb{E}^{\Omega}\,|\,
\exists \bg\in\Omega,\,\forall k\ge\bg,\, \xi_{k}=0\right\} 
$
is the set of all \emph{eventually zero  $\Og$-sequences} of elements 
of $\EE$, and define
the sum map $\Sg:\S(\Og,\EE)\to\EE$ by
$\Sg(\xib)=\sum_{k}\xi_{k}$.
%

We now associate to each $x\in\S(\Og,\R)=\S(\Og,\EE)\cap \R^{\Og}$ an 
$\Og$-sequence of 
real numbers, its \emph{counting function}:\footnote{~
The relevance of the counting functions will result below, when it will become apparent that the \emph{counting function} plays
(for \emph{transfinite sums})
the role played by the sequence of the \emph{partial sums }for the
\emph{usual infinite series}.
On the other hand, the qualification ``counting'' is due to their 
meaning in the theory of numerosities to be  developed in Section \ref{num}.}
\begin{defn}
The \emph{counting function} of the $\Og$-sequence $\xb\in\S(\Og,\R)$ 
is the function $\varphi_{\xb}:\Og\to\R$ such that
\[
\varphi_{\xb}(\ag)=\sum_{k\qincl\ag}x_{k}\ \ 
\mathrm{for\ all}\  \ag\in\Og.
\]
\end{defn}
Given $j\in\Og$ and any set $X$, call a function $\psi:\Og\to X$
\emph{$j$-periodic} if 
\[ 
\psi(2^{j}\odot h+\ag)=\psi(\ag)\ \mathrm{for\ all}\ 
\ag<2^{j}\ \mathrm{and\ all}\ h\in\Og. 
\]

Call $\psi$ \emph{periodic} if it is $j$-periodic for some $j=2^j\in\Og$.

The counting functions are exactly the real valued  periodic functions, namely  
 
\begin{thm}\label{S=B}  ${~}$   
  Every counting function is periodic, and 
 conversely every $ j$-periodic $
\psi\in\R^\Og\,$  is the counting function of
some $\xb\in\S(\Og,\R)$ such that $x_i=0$ for $i\ge j$.


\end{thm}
\begin{proof}
  Let $x_\bg\in\R$ be zero above $j$. Then $\sum_{\bg\qincl\ag}x_\bg =\sum_{\bg\qincl2^j\cg+\ag}x_\bg$, for all $\ag\in C(2^j\cg)$.

Conversely, let $\fg\in\R^\Og$ be $j$-periodic, and define $x_\ag$ inductively on the length of the base-$2$-normal form of $\ag$, by
putting
 $$x_0=\fg(0),\ \ x_{2^\ag}=\fg(2^\ag)-\fg(0),\ \ x_{2^\ag+\bg}= \fg(2^\ag+\bg)-\fg(\bg)-\sum_{\cg\qinclp\bg}x_{2^\ag+\cg}.$$
Thus $\ \ \ x_{2^j+\bg}= \fg(2^j+\bg)-\fg(\bg)-\sum_{\cg\qinclp\bg} [\fg(2^j+\cg)-\fg(\cg)]=0$

\qed

\end{proof}
%

The above theorem has an important consequence:

\begin{thm}\label{comp}
     Let 
$\mathscr{B}(\Og,\,\mathbb{R})$ be the set of all real valued periodic functions, and define the map
$
J:\,\mathscr{B}(\Og,\,\mathbb{R})\rightarrow\mathbb{E}$
by $J(\psi)=
\Sg(\xb),\
$ for any $\xb$ such that $\fg_\xb=\psi$.
Then $J$ is a well defined $\R$-algebra homomorphism onto the ordered field $\EE$. 
\end{thm}
\begin{proof}
    First of all,  by Theorem \ref{S=B}, for any 
    $\psi\in \B(\Og,\R)$ there is an    $\xb\in\S(\Og,\R)$ such that $\fg_\xb=\psi,$ and by \CA\ the value of $\sum \xb$ is uniquely determined by $\psi$.
    Hence the map $J$ is well defined.

    \smallskip
   The map $\Sg$ being $\R$-linear, $J$ preserves linear combinations 
over $\R$. Moreover, by the axiom \RA, the range of $J$ is the whole field $\EE$.

\smallskip
Finally, given $\xb,\yb\in\S(\Og,\R)$ such that $x_h=y_h=0$ for $h\ge i$, put $z_{j}=\sum_{h\vee k= j}x_{h}y_{k}$. Then,
by the Product  axiom \PA, we have $$\Sg(\xb)\cdot\Sg(\yb)=\sum_{h,k}x_hy_k=\sum_jz_j=\Sg(\zb). $$
Hence one obtains, for all $\ag\in\Og,$
$$(\fg_{\xb}\cdot\fg_{\yb})(\ag)=\fg_{\xb}(\ag)\cdot\fg_{\yb}(\ag)=
\sum_{h\qincl\alpha}x_{h}\cdot
\sum_{k\qincl\alpha}y_{k}
=\sum_{h,k\qincl\ag}x_{h}y_{k}=
\sum_{j\qincl\ag}z_{j}=\fg_\zb(\ag)$$

\smallskip 
 
 Therefore also  the products are preserved.
\qed
\end{proof}

\begin{rem}
The kernel of $J$ is a \emph{maximal ideal} determined by its
\emph{idempotents}, 
so there is an \emph{ultrafilter} $\U({\Og})$ on $\Og$ such that
$$ {~~~~~~~~~~~~~~~~~~~~~}\fg\in\ker J\ \ \Iff\ \ \{\ag\in\Og\mid\, 
\fg(\ag)=0\,\}\in\U(\Og),{~~~~~~~~~~~~~~}\mathbf{(U.1)}$$
or equivalently
$${~~~~~~~~~~~~~~~~~~}\sum_{k}x_{k}=0\ \ \Iff\ \ \{\ag\in\Og\mid\, 
\sum_{k\qincl\ag}x_{k}=0\,\}\in\U(\Og).{~~~~~~~~~~~~}\mathbf{(U.2)}$$
This fact will be basic in the constuction of the Euclidean field 
given in Subsection \ref{cons}.
\end{rem}  

\bigskip
We could extend the definition of the counting function $\varphi_{\xib}:\Og\to\EE^{\Og}$ to the set 
$\S(\Og,\EE)$ of all eventually zero
 $\Og$-sequences of Euclidean numbers in the natural way:
%
$$
\varphi_{\xib}(\alpha)=\sum_{k\qincl\alpha}\xi_{k}\ \ 
\mathrm{for\ all}\  \ag\in\Og.
$$

\noindent
The defining sums are ordinary finite sums in the field $\EE$, so 
%
 we might extend the homomorphism $J$ to the whole algebra 
$\A(\Og,\EE)=\{\,\fg_{\xib}\mid\, \xib\in\S(\Og,\EE)\,\}$, and obtain
a $\R$-linear application $
J_{\EE}$ onto the ordered field $\EE$,
such that $$
J_{\EE}(\varphi_{\xib})=\sum_{k}\xi_{k}\ \
 \mathit{for\ all}\ \ \xib\in\S(\Og,\EE).$$

\noindent
\textbf{\emph{CAVEAT.}} The Comparison Axiom \CA\ does not hold for transfinite sums of general Euclidean numbers, so the map $J_{\EE}$ is not an \emph{algebra homomorphism}. The kernel of  $J_{\EE}$ is a \emph{subspace}
of $\A(\Og,\EE)$ such that $(\ker J_{\EE})\cap \B(\Og,\R)=\ker J$.
But it is 
not an \emph{ideal}, \emph{a fortiori} it is not definable through
 an ultrafilter on $\Og$ by extending the conditions $(U.1),(U.2)$ 
 above to transfinite sums of general Euclidean numbers.
 
 \smallskip
 In fact there exist  $\xi_k$ in $\EE$ such that $\xi_k=0$ for $k\ge\og$ and $\sum_{k}\xi_k=0$, while $\xi_k>0$ for all $k<\og$, so any partial sum $\sum_{k\incl\ag}\xi_k$ 
 is \emph{greater than zero}: \pes\ take $\xi_0=\og-1,\ \ \xi_k=1$ for $0<k<\og.$ Then 
 $\sum_{k}\xi_k=\og-1-(\og-1)=0$, but $\sum_{k\incl \ag}\xi_k =\og-2^{|L_n|}>0$ for  $\ag=\og\odot\cg+n$.
\medskip

Given the eventually zero  $\Og$-sequence $\xb=\la x_k\mid k\in\Og\ra\in \S(\Omega,\,\mathbb{R})$, 
its \emph{counting function} $\varphi_{\xb}$ 
 is
an $\Og$-sequence $\varphi_{\xb}(\ag)=\sum_{k\qincl\ag}x_{k}$ of \emph{finite partial sums} from the 
transfinite sum $\sum_{k}x_{k}$, and we would like to write
\begin{equation}
\sum_{k}x_{k}=\lim_{\ag\uparrow\Og}\fg_{x}
(\ag)\label{eq:bella-1-1}
\end{equation}
where the limit should be taken towards 
an appropriate ``point at infinity" $\Og$.

 Call \emph{$\Og$-limit} the limit so defined. Then the following properties 
hold by definition:
\begin{description}
\item[ ($\Og.{1}$)]\ \textbf{Existence and uniqueness:}\\
$\ \ $ Every 
$\varphi\in\B(\Omega,\R)$\ has a unique
$\Og$-limit \ \
$\underset{\eta\uparrow\Og}{\lim}\,\fg(\beta)=\xi\in\EE$, 
\\
and every \emph{\ }$\xi\in\mathbb{E}$ is the $\Og$-limit\ of
some net $\varphi\in\B(\Omega,\R)$\emph{. }
\item[ ($\Og.2$)]\ \textbf{Real numbers preservation}: 
\[ \left(\,\exists\beta_{0}\in\Omega\; \forall\beta\sqsupseteq \beta_{0}\,.\,
\ \varphi(\beta)=r\,\right) \ \Imp\ \ 
\lim_{\beta\uparrow\Og}\varphi(\beta)=r
\]

\item[ ($\Og.3$)]\ \textbf{Sum and product preservation}:
$\ \ $ For all $\fg,\psi\in\B(\Og,\R)$
\begin{eqnarray*}
\lim_{\beta\uparrow{\Og}}\varphi(\beta)+
\lim_{\beta\uparrow{\Og}}\psi(\beta) & = & 
\lim_{\beta\uparrow\Og}\left(\varphi(\beta)+\psi(\beta)\right)\\
\lim_{\beta\uparrow{\Og}}\varphi(\beta)\cdot
\lim_{\beta\uparrow{\Og}}\psi(\beta) & = & 
\lim_{\beta\uparrow\Og}\left(\varphi(\beta)\cdot\psi(\beta)\right)
\end{eqnarray*}

\end{description}

 The  properties
($\Og$.1-3) are assumed
as axioms in \cite{ultra} (with an appropriate directed set $\Lambda$ replacing 
$\Omega$), thus providing a different
approach to Nostandard Analysis, called $\Lambda$-theory, 
usuful for the applications. The theory of the Euclidean
numbers, having more structure, is \emph{a fortiori} suitable to this aim:
the next section is devoted to this developement. However the \emph{product preservation} in $(\Og.3)$ \emph{fails} if extended to arbitrary $\fg,\psi\in\A(\Og,\EE)$, as shown by the example given in the \emph{Caveat } above.

\section{Euclidean numbers and Nonstandard Analysis}\label{NSA}

In this section we show that the Euclidean numbers are hyperreal
numbers, actually they are the \emph{unique saturated field of hyperreal numbers
with the cardinality of }$\Omega$.

\subsection{Hyperreal fields}\label{hyp}
Many different approaches to
Nonstandtard Analysis can be found in the literature, see in particular \cite{rob,keisler76,BDNF2} and the bibliography therein. For completeness, we briefly recall here the basic definitions of the 
superstructure approach.
\begin{defn}
For any set $X$ of atoms, the \emph{superstructure over} $X$ is the set
\[
V_{\omega}(X)=\bigcup\limits _{n\in\mathbb{N}}V_{n}(X)
\]
where
\[
V_{0}(X)=X\ \ 
\mathrm{and}\ \
V_{n+1}(X)=V_{n}(X)\cup\mathcal{P}(V_{n}(X))
\]

\end{defn}
\bigskip{}

\begin{defn}
\label{def:rosa}Given a field $\mathbb{\mathbb{F}}\supset\mathbb{R}$,
a \emph{nonstandard embedding} is a mapping
\begin{equation}
\ast:V_{\omega}(\mathbb{R})\rightarrow V_{\omega}(\mathbb{F});\label{lina}
\end{equation}
 that satisfies the
\emph{Leibniz transfer principle}, \ie\ 
\[
\rho(a_{1},...,a_{n})\Longleftrightarrow\rho(a_{1}^{\ast},...,a_{n}^{\ast})
\]
for all \textit{bounded quantifier
formul\ae}\footnote{
~By \emph{bounded quantifier formula} we mean a first-order formula
in the language $\mathcal{L}=\left\{ \in\right\} $ of set theory,
where all quantifieers occur in the bounded forms $\forall x\in y$
or $\exists x\in y$.
} $\rho(x_{1},...,x_{n})$ and all  $a_{1},...,a_{n}\in
V_{\omega}(\mathbb{R}).$

Moreover, it is assumed that $r^{\ast}=r$ for every $r\in\mathbb{R}$, that 
$\mathbb{R}^{\ast}=\mathbb{F}$, and that
$\mathbb{F}$ is a set of \emph{atoms}.

 Given a nonstandard embedding
 $\ast:V_{\omega}(\mathbb{R})\rightarrow V_{\omega}(\mathbb{F})$,
 the triple $(\ast,\,\mathbb{R},\,\mathbb{F})$ is called \emph{hyperreal
number system}, and the field {$\mathbb{F}$}
is called {\emph{hyperreal field}}.

\end{defn}

\subsection{The Euclidean numbers as hyperreal numbers}\label{euhyp}

In this section we show that $\mathbb{E}$ is a \emph{hyperreal field}
by giving an explicit definition of the map $\ast$ in
(\ref{lina}). This is one of the reasons why we have assumed that the
Euclidean numbers are \emph{atoms}.
\begin{defn}
Given any set $S$, a real algebra of functions $\mathfrak{F}\left(S,\mathbb{R}\right)\subset\mathbb{R}^{S}$
is called \textit{composable} if
\[
\forall f\in\mathbb{R}^{\mathbb{R}}\:\,
\forall\varphi\in\mathfrak{F}\left(S,\mathbb{R}\right)\,.\,
f\circ\varphi\in\mathfrak{F}\left(S,\mathbb{R}\right).
\]

\end{defn}
\smallskip{}

Recall the following theorem of Benci and Di Nasso:
\begin{thm}[\cite{BDN04}, Thm.~3.3]
\label{BDN}A field $\mathbb{F}$ is a hyperreal field if and only
if there exist a set $S$, a composable algebra of functions 
$\mathfrak{F}\left(S,\mathbb{R}\right)\subset\mathbb{R}^{S}$, 
and a surjective homorphism
$
J:\,\mathfrak{F}\left(S,\mathbb{R}\right)\rightarrow\mathbb{F}.
$
\end{thm}
Applying this theorem together with Theorem \ref{comp}, we immediately get:
\begin{thm}
\label{thm:dod}The algebra $\B(\Og,\R)$ is composable, hence
 $\mathbb{E}$ is a hyperreal field.
 ${~}$ \qed
\end{thm}
Now we  define explicitly the map $*$ and some other notions
of Nonstandard Analysis by applying the $\Og$-limit introduced 
in Subsection \ref{count}. 
\begin{defn}
Given a periodic $\Og$-sequence
$\varphi\in\B( \Omega, V_{n}(\mathbb{R}))$, define by induction its 
$\Og$-limit\, 
$
\underset{\eta\uparrow\Og}{lim}\,\varphi(\eta)
$\,
 as follows: 
 \begin{itemize}
     \item  for $n=0$, the limit  
     $\underset{\eta\uparrow\Og}{\lim}\,\varphi(\eta)=J(\fg)$
has been defined by the condition (\ref{eq:bella-1-1}) in Subsection \ref{count};
 
so, assuming the limit defined
for $n$, put, for $\fg\in\B( \Omega, V_{n+1}(\mathbb{R}))$:
     \item $\underset{\eta\uparrow\Og}{\lim}\,\varphi(\eta)=
\left\{ \underset{\eta\uparrow\Og}{\lim}\,\psi(\eta)\ |\ 
\psi\in\B( \Omega, V_{n}(\mathbb{R}))\text{ and}\ 
\forall\eta\in\Omega,\ \psi(\eta)\in\varphi(\eta)\right\} .$
 
\item A set in {$V_{\omega}(\mathbb{E})$}
which is the $\Og$-limit of an $\Og$-sequence is called \textit{internal}. 

\item A \emph{mathematical entity} (number, set, function or relation), when
 identified with a set in {$V_{\omega}(\mathbb{E})$,}
is called \emph{internal} if the corresponding set is internal. 
 \end{itemize}
 \end{defn}
Now we can
define the $*$-map.
\begin{defn}
   If $r\in\mathbb{R}$, then $r^{*}=r.$  If $E\in V_{\og}(\R)$ is a set, then
the \textit{star extension } $E^{\ast}$ of $E$ is 
\[
E^{\ast}:=\underset{\eta\uparrow\Og}{\lim}\ c_{E}(\eta)=
\ \left\{ \underset{\eta\uparrow\Og}{\lim}\ 
\psi(\eta)\ |\ \psi(\eta)\in E\right\} =
\left\{\,
J(\varphi)\ |\ \varphi\in\mathscr{\mathscr{B}}(\Omega,\,E)\right\}
\]
where $c_{E}(\xi)$ is the sequence identically equal to $E$. 
\end{defn}



This appoach to Nonstandard Analysis being based on the notion of limit,
it is natural to formulate the Leibniz principle in the following 
apparently stronger
form:
\begin{prop} $~$
Let $~\rho(x_{1},...,x_{n})~$ be a bounded quantifier formula and let\\
$\varphi_{1},$ \ldots, $\varphi_{n}$ be $\Og$-sequences in 
$\B(\Og,V_{N}(\mathbb{R})), with\ N\in\mathbb{N}$;
then
\[
\left( \exists Q\in\mathcal{U}({\Omega})\,\,\forall\xi\in Q\,.\,
\rho(\varphi_{1}(\xi),...,\varphi_{n}(\xi))\right)\
\Longleftrightarrow\ \rho\left(\underset{\xi\uparrow\Og}{\lim}\,\varphi_{1}(\xi),...,\underset{\xi\uparrow\Og}{lim}\,\varphi_{n}(\xi)\right)
\]
\end{prop}
\begin{proof}
The proof is a simple adaptation of the usual proof (see 
\pes\ \cite{keisler76}).

$~$ \qed
\end{proof}

\medskip
Clearly the Leibniz principle as formulated in Def.~\ref{def:rosa}
follows by taking constant sequences in the above proposition. 
Hence
\begin{cor}
\label{cor:bullo}The Euclidean number field $\EE$ is a real closed field.\end{cor}
\begin{proof}
Since $\mathbb{R}$ is  real closed, then also $\mathbb{\mathbb{E}}=\mathbb{R}^{*}$
is real closed.
\qed
\end{proof}

\subsection{Saturation}\label{sat}

Recall the usual definiton of saturated field in Nonstandard Analysis:
\begin{defn}
A {hyperreal number system} $(*,\,\mathbb{R},\,\mathbb{F})$
is \textit{saturated} if any family of \emph{internal} sets
$\E=\left\{ E_{k}\mid {k\in K}\right\}$ of size  $|K|<|\mathbb{F}|$, with the
\emph{finite intersection property}\footnote{ The \FIP\ says 
thet every finite subfamily of $\E$ has nonempty intersection.} \FIP, has a nonempty intersection.\end{defn}
\begin{thm}
\label{satu}The {Euclidean number system} $(*,\,\mathbb{R},\,\mathbb{E})$
is saturated.\end{thm}
\begin{proof}
Let $\left\{ E_{k}\right\} _{k\in K},\, |K|\!<\!|\Og|,$ be a family of
internal sets with 
the \FIP. Recall that each internal set is the $\Og$-limit of a periodic 
function 
$\fg\in\B(\Og,\P(\R))$. 
The ordinal $\Og$ being 
(strongly) inaccessible, the periods are bounded by
$2^{j}$, for some $j\in\Og$. So we assume 
w.l.o.g. that
 $E_{k}=\lim_{\ag\uparrow{\Og}} E_{k,\ag}$, with
$\R\linc E_{k,\ag}=E_{k,2^jh+\ag}$ for all $k\in K$, all $h\in\Og$, and all $\ag<2^j$.
For sake of simplicity we assume that also  the index set is a 
segment of ordinals, say $K=\{k\in\Og\mid\, k<2^{j}\}$, and we put 
\[
E_{k}=E_{2^jh+k}\ \mathrm{for\ all}\ h\in\Og \ \mathrm{and\ all}\ k\in K
\]
so as to have $E_{k}$ periodically defined for all $k\in\Omega$.

By \FIP, the internal sets $B_{\bg}
:=\bigcap_{k\qincl\bg}E_{k}$ are nonempty for all $\bg\in\Og$.
Remark that we have extended the definition of the sets $E_{k}$ to 
all $k\in\Og$ in such a way that 
$B_{2^jh+\bg}$ coincides with $B_{ \bg}$ for all $\bg<2^j$ and all $h\in\Og$.
The sets $B_{\bg}$ being internal and nonempty, there exists for each 
$\bg$ a family of nonempty sets $\la B_{\bg,\xi}\incl\R\mid \xi\in\Og\ra$ such
that 
\[
B_{\bg}=\lim_{\xi\uparrow\Og}B_{\bg,\xi}.
\]
So for each $\bg$ there exists a sequence 
\[
\varphi_{\bg}\in \B(\Og,\R)\ \ \mathrm{such\ that}\ \ 
\forall \xi\,.\,\varphi_{\bg}(\xi)\in B_{\bg,\xi}
\]
Moreover, by our assumptions, we may choose the functions $\varphi_{\bg}$ 
so that
$$\varphi_{2^jh+\bg}(2^jk+\xi)=\varphi_{\bg}(
\xi)\ \ \forall \bg,\xi<2^j \,\forall h,k\in\Og.$$
Now define $\psi\in \B(\Og,\R)$ by putting, for $s<2^{j}$,
\[
\psi(2^{j}\odot\bg+s):=\varphi_{\bg}\left(s\right),
\ \ \mathrm{so\ that}\ \ 
\forall\bg\in\Omega\,\,\forall s<2^{j}\,.\,
\psi(2^{j}\odot\bg+s)\in B_{\bg,s}.
\]
Then $\psi(\xi)\in B_{\bg,\xi}$ for all $\xi\in\Og$, hence 
$$b=\lim_{\xi\uparrow\Og}\psi(\xi)\in 
\lim_{\xi\uparrow\Og}B_{\bg,\xi}=B_{\bg}\ \ \mathrm{for \ all}\ \ 
\bg\in\Og,$$
and so $b\in             \bigcap_{\bg\in\Og}B_{\bg}=
\bigcap_{k\in K}E_{k}$, which is therefore nonempty.

\qed
\end{proof}

\medskip
An immediate consequence of saturation is that the order-type of $\mathbb{\mathbb{E}}$ is $\eta_\Og$,\footnote{
~An ordered set $X$ has order-type $\eta_\ag$ if, given sets $Y,Z\incl X$, of cardinality less than $\ah_\ag$, s.t.\  $y<z$ for all $y\in Y, z\in Z$,
there exists $x\in X$ s.t.\ $y<x<z$ or all $y\in Y, z\in Z$.}
hence, as ordered field, $\EE$  is \emph{unique} in the following sense:
\begin{cor}\label{unireclo}
If $\ \mathbb{F}$ is a real closed field having order-type $\eta_\Omega$, then $\mathbb{F}$ is isomorphic to $\mathbb{E}$.
\end{cor}
\begin{proof}
Apply the fact that two real closed 
fields of size $\ah_\ag$ are isomorphic if and only if they have
the same \emph{order-type} $\eta_\ag$ (See \pes\ \cite{CK}, p.~348.)
\qed
\end{proof}

\medskip
\begin{rem}
If, instead of the first inaccessible number, we take $\Omega$ to be 
the class of all ordinals,
then the field of Euclidean numbers $\EE$ is isomorphic to
the field of \emph{surreal numbers} \textbf{No} of Conway \cite{con}; so, 
following Ehrlich \cite{el12} they
form an \emph{absolute arithmetic continuum}.
\end{rem}

\subsection{The Keisler hyperreal field}\label{Keis}

Following Keisler we give the following definition of isomorphism 
between hyperreal fields:
.
\begin{defn}
\label{def:KKK}Let $(*,\,\mathbb{R},\,\mathbb{\mathbb{R}}^{*})$
and $(\circledast,\,\mathbb{R},\,\mathbb{R}^{\circledast})$ be 
hyperreal number systems with the same real part $\mathbb{R}$. A map
$
h:\,\mathbb{\mathbb{R}}^{*}\to\mathbb{R}^{\circledast}\,
$ is an
\textit{isomorphism} if the following conditions are fulfilled:
\begin{itemize}
\item (i) $h(r)=r$ for each $r\in\mathbb{R}$,
\item (ii) h is an \emph{ordered field isomorphism} from 
{$\mathbb{R}^{*}$}
onto $\,\mathbb{R}^{\circledast},$
\item (iii) For each real function $f$ of $n$ variables and all 
$x_{1},...,x_{n}\in\mathbb{R}^{*},$
\[
f^{\circledast}(h(x_{1}),...,h(x_{n}))=h\left(f^{*}(x_{1},...,x_{n})\right)
\]

\end{itemize}
Two hyperreal number systems are  \emph{isomorphic} if there is
an isomorphism between them.
\end{defn}


In the set theory ZFC plus the Axiom of Inaccessibility, one can 
prove the following theorem:
\begin{thm}[\cite{keisler76}, p.196]
There is a definable\footnote{
~Recall that
a set $X$ is (first order) definable if there is
a first order formula $\rho(x)$ such that $X$ is the 
unique set such that $\rho(X)$ holds. 
} hyperreal number system $(\bullet;\mathbb{R};\,\mathbb{R}^{\bullet})$
which is saturated and such that the cardinality
of \foreignlanguage{english}{\textup{$\mathbb{R}^{\bullet}$}} is
the first uncountable inaccessible cardinal.\qed
\end{thm}

We shall refer to the field $\mathbb{R}^{\bullet}$ as the
{\emph{hyperreal Keisler field}}.
According to Thm.~\ref{satu}, we have the following
interesting result: 
\begin{cor}\label{unihyp}
The Euclidean number system $(\ast,\,\mathbb{R},\,\mathbb{E})$ is
isomorphic to the hyperreal Keisler field 
$(\bullet;\,\mathbb{R};\,\mathbb{R}^{\bullet})$.
\qed
\end{cor}

\subsection{$\Omega$ versus $\mathbb{N}^{*}$}

As an ordered field, the Euclidean field $\mathbb{E}$
is unique, and it defines a \emph{hyperreal number system}
$(\ast,\,\mathbb{R},\,\mathbb{E})$
which is unique up to isomorphism, by Corollary \ref{unihyp}.
 In addition,
the Euclidean field $\EE$ has two main extra properties which are not
shared by other hyperreal fileds:
\begin{itemize}
\item \emph{the sum of infinitely many hyperreal numbers is well defined};
\item \emph{the semiring of the accessible ordinals $\Omega$, with the \emph{natural sum and product}, is isomorphically embedded in $\mathbb{E}$
in a natural way}.
\end{itemize}
The combination of these features creates new phenomena which
we now investigate. 
\begin{rem}
\label{prop:zerlina} According to Theorem \ref{ordeu}, we have identified each ordinal $\ag\in\Og$ with the Euclidean number given by the natural embedding 
$
\Psi\,:\ \Omega\to\EE
$
defined by 
$$
\Psi(\alpha)=\sum_{k}\chi_{[0,\alpha)}(k)
=\underset{\xi\uparrow\Og}{\lim}\,\varphi_{\alpha}(\xi)
$$
where
\begin{equation}\label{eq:zerlina}
\varphi_{\alpha}(\xi)=\sum_{k\qincl\xi}\chi_{[0,\ag)}(k)=|
\{ k<\alpha\mid\,k\qincl\xi\,\}|.
\end{equation}
Actually, $\Psi$ is an isomorphic embedding of $\Og$ into $\kN$ (as ordered semirings):
in fact, by definition,
\[
\mathbb{N}^{*}=\left\{ \underset{\xi\uparrow\Og}{\lim}\,\varphi(\xi)\ |
\ \varphi\in\mathscr{B}(\Omega,\,\mathbb{\mathbb{N}})\right\}
\incl\EE
\]
so $ \Psi[\Omega]\incl
\mathbb{N}^{*},
$
because the  counting function $\fg_{\ag}$ of
$
\Psi(\alpha)=\sum_{k}\chi_{[0,\alpha)}(k)
$ takes its values in $\N$.
\end{rem}

\medskip
Thus the ordinal numbers in $\Omega\incl\EE$ can be viewed as ``special''
\emph{hypernatural numbers}. In order to investigate the relation of 
general hypernatural numbers with  ordinal numbers, put, for $\xi\in\Og$,
\[
K(\xi)=\left\{ \ag\in\Omega\ |\ \ag\qincl\xi\right\}, 
\ \ \  m(\xi)=|K(\xi)|=2^{|L_\xi|},
\ \
\mathrm{and} \ \  \mathbb{N}_{m}=\left\{ n\in\mathbb{N}\ |\ n< m\right\}.\]
%
Let  $j_{\xi}:\mathbb{N}_{m(\xi)}\to K(\xi)$ be the order-preserving bijection;
in particular
$
j_{\xi}(0)=0$ and $j_{\xi}(m(\xi)-1)=\xi$ for all $\xi\in\Og$.

Take the $\Og$-limits\
\begin{equation}
K:=\underset{\xi\uparrow\Og}{\lim}\,K(\xi)\incl\Og^*,\ \
     \mu=\underset{\xi\uparrow\Og}{\lim}\,m(\xi)\,,\ \ 
  \mathbb{N}_{\mu}^{*}:=\underset{\xi\uparrow\Og}{\lim}\,\mathbb{N}_{m(\xi)}\,,
  \ \ \mathrm{and}\ \ 
   j:=\underset{\xi\uparrow\Og}{\lim}\,j_{\xi}\,.\  \label{eq:pina}
\end{equation}

Then, by Leibniz principle,  
\[
\Omega\subset K
\incl\EE, 
\ \ \
\mathbb{N}_{\mu}^{*}=\left\{ k\in\mathbb{N}^{*}\ |\ 
k<\mu\right\}, \ \ \   \mathrm{and} \ \ \ 
 j:\ \mathbb{N}_{\mu}^{*}\rightarrow K
\] 
is\ an\ order-preserving\ surjection that is the identity when restricted to $\Og$.

\subsection{Hyperfinite sums v/s transfinite sums}\label{hyperf}
In Nonstandard Analysis one deals with particular infinite sums, usually 
indexed by closed initial segments of the set $\kN$ of the nonstandard
natural numbers. We consider in this subsection the relations between these
\emph{hyperfinite} sums of the NSA and the \emph{transfinite} sums of the
Euclidean field $\EE$ introduced in Section \ref{field}.
\begin{defn}
An internal set $F\in V_{\omega}(\mathbb{E})$ is called \emph{ hyperfinite}
if it is the $\Og$\emph{-limit of finite sets}, namely 
\[
F:=\underset{\xi\uparrow\Og}{\lim}\,F_{\xi}=\left\{ \underset{\xi\uparrow\Og}{\lim}\,x_{\xi}\ |\ x_{\xi}\in F_{\xi}\right\} 
\]
where the $F_{\xi}$s are finite sets in $V_\og (\R)$. 
\end{defn}
\Pes, for $\mu=\lim_{\xi\ua\Og} m_{\xi}\in\kN $, the set\
$\
\mathbb{N}_{\mu}^{*}=\left\{ \nu\in\mathbb{N}^{*}\ |\ \nu<\mu\right\} \ $
\
{is hyperfinite, since}\ \ \ 
$\mathbb{N}_{\mu}^{*}=\underset{\xi\uparrow\Og}{\lim}\left\{ n\in\mathbb{N}\ |\ n< m_{\xi}\right\}. 
$

The notion of hyperfinite set is basic in defining the  notion 
of \emph{hyperfinite sum}, which is the usual notion of infinite sum of hyperreal numbers:
\begin{defn}
Given a hyperfinite set of hyperreal numbers $F\subset\mathbb{E}$, the
\emph{hyperfinite sum of the elements o}f $F$ is defined as follows:
\[
\sum_{x\in F}^{*}x=\underset{\xi\uparrow\Og}{\lim}\left(\sum_{x_\xi\in F_{\xi}}x_\xi\right)
\]

Transfinite sums and hyperfinite sums are strictly related,
as the next theorem shows.
\end{defn}
\begin{thm}
If $\ab\in\mathscr{S}(\Omega,\mathbb{R})$, then
\[
\sum_{k}a_{k}=\sum_{x\in F^{a}}^{*}x
\]
where the hyperfinite set $F^\ab$ is defined as
\begin{equation}
F^{\ab}=\underset{\xi\uparrow\Og}{\lim}\ F_{\xi}^{\ab}\ \ \ where\ \ \ 
F_{\xi}^{\ab}=\left\{ a_{k}\ |\ k\qincl\xi\right\} \label{eq:pesce}
\end{equation}
\end{thm}
\begin{proof}
We have
\[
\sum_{k}a_{k}=\underset{\xi\uparrow\Og}{\lim}\,\sum_{k\qincl\xi}a_{k}=\underset{\xi\uparrow\Og}{\lim}\,\sum_{x_\xi\in F_{\xi}^{\ab}}x_\xi=\sum_{x\in F^{\ab}}^{*}x.
\]
\qed
\end{proof}
\medskip{}
By the Leibniz principle, we have that a set $F\subset\mathbb{E}$
is hyperfinite if and only if there exist $a\in\mathbb{R}^{\mathbb{N}^{*}}$
and $\mu\in\mathbb{N}^{*}$ such that
\[
F=F_{\mu}^{a}:=\left\{ a_{\nu}\,|\,\nu\in\mathbb{N}^{*},\nu<\mu\right\} .
\]
This fact suggests the following notation:
\[
\sum_{\nu\in\mathbb{N}_{\mu}^{*}}a_{\nu}=\sum^*_{x\in F_{\mu}^{a}}x.
\]
Given a sequence $a\in\mathbb{R}^{\mathbb{N}}$,  put $S(n)=\sum_{k\qincl n}a_{k}$. Denote by
$a^{*}, S^*$
 the $*$-extensions 
of $a,S$ respectively: then,
 for any hypernatural number $\mu\in\mathbb{N}^{*}$,  the corresponding
hyperfinite sum is
\[
\sum_{\nu\in\mathbb{N}_{\mu}^{*}}^{*}a_{\nu}^{*}=S^{*}(\mu),
\]
by  the Leibniz principle.
In particular we have\ \ 
\begin{thm}\[
\sum_{k<\omega}a_{k}=\sum_{\nu\in\mathbb{N}_{\omega}^{*}}^{*}a_{\nu}^{*}\ \ \ for\ all\ \ \ a\in\mathbb{R}^{\mathbb{N}}.
\]
\end{thm}
\begin{proof}
Let $\varphi$ be the counting function of $\ \sum_{k<\omega}a_{k}$,
so
\[
\varphi(\og\odot h+n)=\sum_{k\qincl\og\odot h+n}a_{k}\chi_{\omega}(k)=\sum_{k\qincl n}a_{k}=S(n)
\ \ \ \mathrm{whence}
\ \ \
\sum_{k<\omega}a_{k}=\underset{\og\odot h+n\uparrow\Og}{\lim}\,S(n).
\]
We have taken $\og=\sum_k \chi_{\og} (k)=
\underset{\og\odot h+n\uparrow\Og}{\lim}\,n:
$
then we have that
\begin{align*}
\sum_{k<\og}a_{k} & =S^{*}\left(\underset{\og\odot h+n\uparrow\Og}{\lim}\, n\right)
  =S^{*}(\omega)=\sum_{\nu\in\mathbb{N}_{\omega}^{*}}a_{\nu}^{*}
\end{align*}
\qed
\end{proof}
\medskip{}
\Pes,  compute the sum $\sum_{k\in\N}
\,\frac{1}{2^{k}}
$:
\begin{equation}
S(n)=\sum_{k<n}\,\frac{1}{2^{k}}=\frac{1-\frac{1}{2^{n}}}{1-\frac{1}{2}}=2-\frac{1}{2^{n-1}},
 \ \ \mathrm{hence}\ \ 
\sum_{k<\omega}\,\frac{1}{2^{k}}=2-\frac{1}{2^{\omega-1}}\label{eq:marta}
\end{equation}
where $\og$ is considered as Euclidean integer, and so the power $2^{\og-1}$  is computed in the nonstandard sense.

This computation is ``more accurate" than the usual value given to the series
$
\sum_{k=0}^{+\infty}{2^{-k}}=2,
$
which neglects the infinitesimal $2^{1-\og}$ in (\ref{eq:marta}).

\section{Numerosities}\label{num}
In the history of Mathematics the problem of comparing the size of
objects has been extensively studied. In particular different methods
of measuring sets, by associating to them suitable kinds of numbers,
have been exploited. 
 A satisfactory notion of measure for sets
should be submitted to the famous five \emph{common notions} 
of Euclid's Elements, which traditionally 
embody the properties of magnitudes (see \cite{Eu}):

\begin{enumerate}
\item
\emph{Things equal to the same thing are also equal to one another.}
\item
\emph{And if equals be added to equals, the wholes are equal.}
\item
\emph{And if equals be subtracted from equals, the remainders are equal.}
\item
\emph{Things \emph{[exactly] applying onto} one  another are  equal to one
another.}\footnote
{~Here we translate $\epsilon\phi\ag\rho\mu o\zeta o\nu\tau\!\ag$
by ``[exactly] applying onto", instead of the usual ``coinciding with". 
As pointed out by T.L.~Heath in his commentary \cite{Eu}, 
this translation seems to give a more appropriate
rendering of the mathematical usage of the verb
$\epsilon\phi\ag\rho\mu o\zeta \epsilon\iota\nu$.}
\item
\emph{The whole is greater than the part.}
\end{enumerate}

To be sure, the Euclidean common notions seem \emph{prima facie} 
unsuitable for measuring the size of arbitrary sets: the third and fifth notions are known 
to be incompatible with the very ground of the Cantorian theory of 
cardinality, namely the so called \emph{Hume's Principle}

 \begin{description}
     \item[$(\mathsf{{HP}}) $]   \emph{Two sets 
  have the same size if and only if there exists a biunique 
  correspondence between  them.}
 \end{description}

This principle amounts to encompass the \emph{largest} possible class of 
size preserving applications, namely \emph{all bijections}. This 
fact might seem natural, and even \emph{implicit in the notion of 
counting}; 
but it strongly violates the equally natural \emph{Euclid's principle}
 \begin{description}
     \item[$(\mathsf{{EP}})$]   
    \emph{A set 
     is greater than its proper subsets,}
 \end{description}
which in turn seems \emph{implicit in the notion 
of size}. Be it as it may, the spectacular development of set theory 
in the entire twentieth century has put Euclid's principle in 
oblivion. Only the new millennium has seen a limited resurgence of 
proposals including \EP\ at the cost of some limitations of \HP\ 
(see \pes\ \cite{benci95b,BDNlab,BDNF1,BDNFuniv}, or the excellent survey 
\cite{mancu} and the references therein).

\subsection{Saving the five Euclidean common notions}\label{saveu}

It is worth noticing that taditional geometry
 satisfies the  Euclidean
common notions because there is a restricted class of ``exact 
applications'': \pes, the \emph{rigid equidecompositions} of polygons, or, more 
generally, when considering metric spaces,  the 
\emph{isometries}, \ie\ distance preserving bijections. 
But, when dealing with general set theory, an appeal to the notion of 
distance seems inappropriate.
So the question arises as to \emph{which correspondences }can be taken as ``exact 
applications'' in order to fulfill the five Euclidean common notions.

Cantor himself, besides his cardinality theory based on general 
bijections, introduced another way of assigning  numbers to sets, 
namely refined \emph{cardinal} numbers to \emph{ordinal} numbers. In this case 
one considers sets 
endowed with a  \emph{wellordering}, and 
restricts the ``exact 
applications'' to the order preserving 
bijections.
However, while
the ordinal 
arithmetic may respect the third common notion, nevertheless the 
Euclid's principle $\EP$ still badly fails.

The \emph{Euclidean} numbers  have been introduced above as an \emph{extension} of 
the ordinals, suitable to provide a notion of size satisfying
\emph{all} the Euclidean
common notions for an appropriate class of ``labelled sets''.  A 
\emph{labelled
set}
$E$ comes together with a suitable \emph{labelling map} 
$
\ell
$
such that $\ell^{-1}(x)$ is a \emph{finite set} for all $x\in E$.   

The original idea is that by putting
 an appropriate \emph{labelling}  on arbitrary sets, 
 the \emph{label preserving bijections} (intended as ``exact 
 applications'') might be used in defining an appropriate Euclidean notion of size, that produces exactly the 
``nonnegative integers'' of the Euclidean numbers (whence their 
name).\footnote{~
 A notion of \emph{numerosity} for
countable ``labelled sets",  
whose elements come with suitable labels (given by natural numbers)
was first presenteded in \cite{benci95b}, and later developed
in the paper \cite{BDNlab}. It provides
a notion of ``number of
elements'' that fulfills the fifth Euclidean common notion, and 
produces particular 
\emph{nonstandard integers}.
}
Remark that a labelled set can be viewed as a generalization of a 
wellordered set, because the latter can be naturally labelled by 
the unique order isomorphism 
with an (initial segment of an) ordinal. In fact we shall see below that 
any labelled set is ``equinumerous'' to a \emph{set of ordinals}.

Let us state the basic definitions.
\begin{defn}
\label{inset}A \emph{labelled set} is a pair $\left(E,\ell\right)$,
where 
\begin{itemize}
\item $E$ is a set of cardinality less than $\Og$, the set of the 
 ordinals smaller than the first inaccessible cardinal;\footnote{
 ~We consider the ordinals in $\Og$ as \emph{Euclidean numbers} in 
 $\EE$, hence  atoms, according to our stipulation following 
 Thm.\ref{ordeucl}; however, as far as the numerosity 
 theories are concerned, the Von Neumenn ordinals would be equivalent. } 
\item $\ell:E\rightarrow\Omega$\
is a function (the \emph{labelling function}) such that 
\begin{enumerate}
\item \label{enu:set-1} the set $\ell^{-1}(x)$ is finite for all $x\in\Omega$, 
\item \label{enu:eti-1} $\ell(x)=x$ for all $x\in E\cap\Omega$. 
\end{enumerate}
\item Two labelled sets $\left(E_{1},\ell_{1}\right)$ and $\left(E_{2},\ell_{2}\right)$
are \emph{coherent} if 
\[
\ell_{1}(x)=\ell_{2}(x)\ \ \mathrm{for\ all}\ x\in E_{1}\cap E_{2}.
\]
\item Two labelled sets $\left(E_{1},\ell_{1}\right)$ and $\left(E_{2},\ell_{2}\right)$
are \emph{isomorphic} if there is a biunique map $\phi:\ E_{1}\rightarrow E_{2}$
such that 
\[
\,\ell_{2}(\phi(x))=\ell_{1}(x) \ \  \ \mathrm{for\ all}\ x\in E_{1}.
\]
\end{itemize}
\end{defn}
Now we are ready to introduce our main notion: 
\begin{defn}
\label{cannum} The \emph{Euclidean} \emph{numerosity} of a labelled
set $(E,\ell)$ is the number 
\[
\mathfrak{n}(E,\ell)=\sum_{k}|\ell^{-1}(k)|\in\mathbb{E}.
\]

\end{defn}
The Euclidean numerosity satisfies \emph{the five Euclidean common notions}
(whence the name), when they are interpreted in the natural
way:

1. Two labelled sets are considered \emph{equal} (in size) if they have the
same numerosity;

2. The \emph{addition} of two labelled sets $\left(E_{1},\ell_{1}\right)$
and $\left(E_{2},\ell_{2}\right)$ with $E_{1}\cap E_{2}=\emptyset$
is given by
\[
\left(E_{1}\cup E_{2},\ell_{1}\cup\ell_{2}\right)\ \ \mathrm{where}\ \ \left(\ell_{1}\cup\ell_{2}\right)(x)=\begin{cases}
\ell_{1}(x) & if\,\,x\in E_{1}\\
\ell_{2}(x) & if\,\,x\in E_{2}
\end{cases}
\]

3. The \emph{subtraction} of two coherent labelled sets $\left(E_{1},\ell_{1}\right)$
and $\left(E_{2},\ell\right)$ with $E_{1}\subset E_{2}$ is given
by
\[
\left(E_{2}\setminus E_{1},\ell_{|E_{2}\setminus E_{1}}\right)
\]

4. Two labelled sets\emph{ [exactly] apply onto} one another is equivalent to
say that they are isomorphic;

5. A \emph{part} of a labelled set is just a (coherent) subset.


\subsection{The Euclidean numerosity theories}\label{eunum}

There are three main \emph{operations} which  produce (possibly new) labelled
sets:
\begin{defn}
\label{def:basop}The \emph{basic operations} on labelled sets
are the following:
\begin{enumerate}
\item \textbf{Subset -} A \emph{subset} of a labelled set $\left(E,\ell\right)$
is a labelled set $\left(F,\ell_{|F}\right)$ where $F\subset E$; 
\item \textbf{Union -} The \emph{union} of two \emph{coherent} labelled sets $\left(E_{1},\ell_{1}\right)$,
$\left(E_{2},\ell_{2}\right)$ is the labelled set 
\[
\left(E_{1}\cup E_{2},\ell\right)\ \ \mathrm{where}\ \ \ell(x)=\begin{cases}
\ell_{1}(x) & if\,\,x\in E_{1}\\
\ell_{2}(x) & if\,\,x\in E_{2}
\end{cases}
\]

\item \textbf{Cartesian product -} The \emph{Cartesian product} of two labelled sets $\left(E_{1},\ell_{1}\right)$,
$\left(E_{2},\ell_{2}\right)$ is the labelled set 
\[
\left(E_{1}\times E_{2},\ell\right)\ \ \mathrm{where}\ \ \ell(x_1,x_2)=
\ell_{1}(x_1) \vee
\ell_{2}(x_2) 
\]

\end{enumerate}
\end{defn}
\medskip{}

\begin{defn}\label{eunumth}
${}$
\begin{itemize}
    \item A
family $(\mathcal{A},\ell)$ of (accessible) \emph{pairwise coherent} labelled 
sets is \emph{closed} if it is closed under the three
basic operations of Def. \ref{def:basop}
  \item  
 a \emph{Euclidean numerosity theory} is a pair $\left(\mathbb{U},\,\mathfrak{n}\right)$, where
$\mathbb{U}$ is a \emph{closed family} of labelled set and
$
\mathfrak{n}: \mathbb{U\rightarrow\mathbb{E}}
$
is the \emph{Euclidean numerosity}. 
\end{itemize}
\end{defn}

\begin{rem}
 Given a coherent family 
$(\mathcal{A},\ell)$ of labelled sets, there exists a {\em least closed family} including 
$\A$, denoted by 
$\UU(\A,\ell)$ and called  the \emph{closure of }$(\A,\ell)$ (or the closed family \emph{generated by} $(\A,\ell)$). 
We omit the labelling function $\ell$ if it  is clear from the context.

 In particular, if $\A=\{X\}$, then we write $\UU[X]$ for $\UU(\A)$. We also write $\UU(\Og)$ for $\mathbb{U}(\{\Og_{\ag}\mid 
\ag\in\Omega\}).$
 \end{rem}
\medskip
Let us see some examples.  Recall that we identify
the natural numbers with the finite ordinal numbers, and the 
accessible ordinals 
 with the corresponding Euclidean numbers.
\begin{itemize}
\item Let $F\subset\Omega$ be a \emph{finite} set, then $\mathbb{U}[F]$ contains
only finite sets and $\mathfrak{n}(E)$ is just the \emph{cardinality} of
the finite set $E$; in this case
\[
\mathfrak{n}(\mathbb{U}[F])=\mathbb{N}\subset\mathbb{E};
\]

\item The ``simplest" numerosity theory  containing \emph{infinite} sets is given
by $\left(\mathbb{U}[\mathbb{N}],\,\mathfrak{n}\right);$ in this
case we have that
\[
\mathfrak{n}(\mathbb{U}[\mathbb{N}])\subseteq
\left\{ \phi(\omega)\in\mathbb{E}\,|\ \phi\in\mathit{\mathscr{B}}
(\Omega,\mathbb{N}),\ \phi(\og h+n)=
\phi(n)\right\} ;
\]

\item The \emph{canonical numerosity theory} is given by 
$(\UU(\Og),\nk)=\left(\mathbb{U}(\{\Og_{\ag}\mid 
\ag\in\Omega\}),\,\mathfrak{n}\right):$
this is the ``simplest" theory which contains all the (accessible) ordinal
numbers.

\end{itemize}

Let us state the main properties of a \emph{Euclidean numerosity
theory.}
\begin{thm}
\label{numeucl} Let $(\mathbb{U},\mathfrak{n})$ be
a Euclidean numerosity theory. Then 
\begin{itemize}
\item each set in $\mathbb{U}$ is equinumerous to
a set of  ordinals, and one has, for all $A,B\in\mathbb{U}$: 
\begin{description}
\item [Sum-Difference] $~~~~~~\bullet~\nk(A\cup B)=\nk(A)+\nk(B)-\nk(A\cap B);$ 
\item [Part-Whole] $~~~~~~~~~~~\bullet~A\pincl B\ \Imp\ \nk(A)<\nk(B);$
\item [Cartesian Product]  $~\bullet~\nk(A\times B)=\nk(A)\cdot\nk(B);$
\item [Comparison]  $~~~~~~~~~~\bullet~$if the $1$-to-$1$ map $T:A\to B$  preserves labels,
 ${~~~~~~~~~~~~~~~~~~~~~~~~~~~~~}then~\nk(A)\le\nk(B)$.
\end{description}
Moreover
\item  If $~\Og_{\ag}\in\A$ for all $\ag\in\Og$, then
 the set of numerosities $\Nk=\nk[\mathbb{U}(\A)]$
is a positive subsemiring of nonstandard integers, that generates
the whole $\mathbb{Z}^{*}$, and one has
\begin{description}
\item[Identity] ${~~~~~~~~~~~~~~~~~~~}\bullet\  \forall\alpha\in\Omega,\,\,\mathfrak{n}(\Omega_{\alpha})=\alpha$;
\item[Cartesian product] ${~~~~~}\bullet \forall\alpha,\beta\in\Omega,\,\,\mathfrak{n}(\Omega_{\alpha}\times\Omega_{\beta})=\alpha\beta$;
\item[Translation invariance] $\bullet\ \forall E\incl\Omega_{\thg},\,\,\mathfrak{n}\left(\{\,2^\thg+\xi\mid\xi\in E\}\right)=\nk(E)$;
\item[Homothety invariance] $\bullet\ \forall E\incl\Omega_{\thg},\,\,\mathfrak{n}\left(\{2^\thg\xi\mid\xi\in E\}\right)=\nk(E)$. 
\end{description}
\end{itemize}
\end{thm}
\
\begin{proof}
    
 All the assertions are straightforward consequences of the 
 definitions and of the axioms of $\EE$, except 
 the first one.
 In order to prove that each set $E\in\mathbb{U}(\mathcal{A})$ is equinumerous to
some set of ordinals, let $\thg$ be an  ordinal greater than all labels 
of the set 
 $E\in\UU$: we define a suitable subset of $\Og_{2^\thg}$ 
 equinumerous to $E$. 
 
 Let $n_{k}$ be the number of elements of $E$ 
 with label $k$, so $\nk(E) =\sum_{k<\thg}n_{k}$.
 
 Let  
$a_{kh}=c_{2^\thg\odot k+h}=\begin{cases}
1\!\!&if\,\,\,\,0\le  h<n_{k} \ and\  k<\thg\,\\
0\!\!&otherwise
\end{cases}.$
 
 Then $\sum_{h}a_{kh}=n_{k}$ for all $k$, and 
 $\sum_{l}c_{l}=\sum_{k}\sum_h a_{kh}=\sum_{k}n_{k}$. Clearly the latter is the numerosity of the set 
 $$L=\{l<2^\thg\thg\mid l=2^\thg k+h,\, h<n_{k},\, n_{k}>0\},$$
 and so $\nk(E)=\nk(L)$.
 
 In particular, every transfinite sum $\sum_{k}n_{k}$ of nonnegative integers is the 
 numerosity of a subset of some $\Og_{\ag}$. Then Theorem \ref{S=B}
  yields that any 
 periodic
 $\Og$-sequence in $\B(\Og,\Z)$ is the difference between the counting 
 functions of two sets of ordinals. Hence $\nk[\UU(\Og)]$ generates 
 the whole $\kZ$.
   
    \qed
\end{proof}

\begin{rem}
Another interesting operation on labelled sets is
\begin{itemize}
\item  \textbf{Finite parts} - 
The \emph{set of the finite parts }of the labelled set
$\left(E,\ell\right)$
is the labelled set 
$\left(\mathcal{P}_{\omega}(E),\,\vee\ell\right)$,\footnote{~
Here $\mathcal{P}_{\omega}(E)$ denotes the family of all finite subsets 
of $E$, also denoted by $[E]^{<\og}$.}
  where $\vee\ell$ is defined as follows:
\[
\vee\ell\left(\left\{ a_{1},...,a_{n}\right\} \right)=\bigvee_{k=1}^{n}\ell(a_{k})
\]
\end{itemize}

If, as usual in axiomatic set theory, we identify the \emph{ordered pair} $(a,b)$ with the ``Kuratowski doubleton" $\left\{ \{a\},\left\{ a,b\right\} \right\} $,
then, by the above definition,
\[
\vee\ell(a,b)=\vee\ell\left(\left\{ \{a\},\left\{ a,b\right\} \right\} \right)=\ell(a)\vee\ell(b).
\]
Hence, if $\left(E_{1},\ell_{1}\right)$ and $\left(E_{2},\ell_{2}\right)$
are labelled sets, their Cartesian product is precisely the labelled set
$
\left(E_{1}\times E_{2},\ell_{1}\vee\ell_{2}\right).
$

It is easily seen that, if the family $\UU$  of the Euclidean numerosity theory $(\UU,\nk)$ is closed also under the operation of finithe parts, then one has
\begin{description}
\item[Finite parts] ${~~}\bullet\  \nk(\mathcal{P}_{\omega}(A))=2^{\nk(A)}.$
\end{description}
\end{rem}

\subsection{Euclidean numerosity v/s Aristotelian size}\label{eucomp}
The Eucldean numerosity theory $(\UU(\Og),\nk)$ of the preceding section might be compared with the ``Aristotelian" numerosity theory introduced in \cite{BDNF1}, where in particular every set of ordinals $A$ receives a numerosity $\sk(A)$ belonging to the non-negative part of an ordered ring $\Ak$ so that the following conditions are fulfilled: 
\begin{description}
\item[$\sp$ ]  If $A\cap B=\0,$ then $\sk(A\cup B)=\sk(A)+\sk(B);$
\item[$\up$]  $~\sk(\{\xi\})=1$ for all $\xi$;
\item[$\fpp$] If $A\incl \Og_{\thg^\cg}$ and $B\incl\Og_{\thg^h}$ with  $\thg=2^\thg$ and $h<\og$,
 then $$\sk(A)\cdot \sk(B)=\sk(\{\thg^\cg\odot \bg+\ag\mid \ag\in A, \bg\in B\}).$$
(For convenience we use here the notation of this article. Moreover we restrict ourselves to sets of \emph{accessible ordinals} less than $\Og$, so as to avoid the use of \emph{proper class-functions}. )
\end{description}

Theorem \ref{numeucl} above immediately implies that these properties are fulfilled by the Euclidean numerosity. 
Conversely, one can define  a \emph{transfinite sum  of non-negative
integers }in the ring $\Ak$, by  assigning to any such sum an appropriate set of ordinals, following the proof of 
Theorem \ref{numeucl}:
$$\sum_{k<\thg}n_k=\sk(L),\ \  where\ \ 
L=\{2^\thg\odot k+m\mid 0\le m<n_k\ne 0,\, k<\thg\}$$

This transfinite sum may be uniquely extended to arbitrary integers by
considering separately positive and negative summands.
We may assume w.l.o.g. that the ring $\Ak$ is generated by the set of 
the numerosities, \ie\ that any $\ak\in\Ak$ is the \emph{difference} 
$\sk(A)-\sk(B)$  of the numerosities of two sets of ordinals: 
then  $\Ak$  becomes the set of all transfinite sums of integers.

It turns out that the axiom \CA\  holds in $\Ak$ for the transfinite sums of integers.
In fact, put $A_{\ag}=\{x\in A\mid x\qincl\ag\}$:
then $|A_{\ag}|=|B_{\ag}|$ for all $\ag\lincq\bg$ implies 
$A\7 A_{\bg}=B\7 B_{\bg}$ and $|A_{\bg}|=|B_{\bg}|$, hence 
$\sk(A)=\sk(B)$.

Moreover the property $\fpp$ can be used in defining the numerosity of Cartesian products. Then a natural strengthening of the property $\fpp$ could be the assumption that  $\sk$ is definable on  \emph{arbitrary sets of pairs} of ordinals by putting
$$\sk\left(E\right)=
\sk(\{2^\thg\cdot \bg+\ag\mid  (\ag, \bg)\in E\})\ \ \rm{for \  all}\ \  E\incl\Og_{\thg}\times \Og_{\thg}\footnote{~Notice that we use here the natural product of ordinals, which agrees with the product of the field $\EE$, in order to avoid  \emph{absorption phaenomena}.
}
$$
On the other hand, the natural labelling of pairs given in the preceding subsection would give to $E$ the numerosity
$$\sk(E)=\sum_k |E(k)|\ \ \ \mathrm{where}\ \ \
E(k)=\{(\ag,\bg)\in E\mid \ag\vee\bg=k\},
$$
\ie, according to the above definition,
$$\sk(E)=\sk(\{2^\thg\cdot k+m\mid  0\le m<\, |E(k)|, k<\thg\}.$$
So we are led to the following 
\begin{defn}\label{mult}
Call \emph{multiplicative} an Aristotelian numerosity $\sk$ such that  following equality holds
for all 
$E\incl\Og_{\thg}\times \Og_{\thg}$
\begin{description}
\item[$\dsp$] $~\sk(\{2^\thg\cdot \bg+\ag\mid  (\ag, \bg)\in E\})=
\sk(\{2^\thg\cdot (\ag\vee \bg)+m\mid  0\le m<|E(\ag\vee\bg)|\}.
~~~~~~~$
\end{description}
\end{defn}
Clearly, the equalities $\dsp$ allow for consistently extending $\sk$ to all sets of pairs of ordinals,  obtaining in particular that $\sk(E\times F)=\sk(E)\cdot\sk(F)$.

Then we have

\begin{thm}\label{compeucl} Let $(\mathbb{U},\sk)$
be a multiplicative Aristotelian numerosity.

Let $\psi$ map each set $E\in\UU$ to the counting function $\psi_E\in \B(\Og,\N)$ such that $\psi_E(\ag)=|E_{\ag}|$, where $E_{\ag}=E\cap\{\bg\in\Og\mid\bg\qincl\ag\}$. 

Let $\Ik$ be the kernel of the homomorphism $J:\B(\Og,\R)\to \EE$ of Thm. \ref{comp}, and let 
$\ik=\Ik\cap \B(\Og,\Z)$ be the restriction of $\Ik$ to
$\B(\Og,\Z)$. 

Then $\ik$ is  generated by the differences $\psi_E-\psi_F$ with $\sk(E)=\sk(F)$, and 
 there exists a unique ordered ring  isomorphism \ $\sg:\Ak\to\,\kZ\incl\EE$
 that makes the following diagram commute\footnote{~here $\imath$ is the inclusion map, and $\pi_{\ik}$ and $\pi_{\Ik}$ are the projections onto the quotients modulo the ideals  $\ik$ and $\Ik$, respectively,}

\bigskip
\begin{center}
\begin{picture}(185,70)
   \put(0,0){\makebox(0,0){$\Ak$}}
   \put(90,0){\makebox(0,0){$\B(\Og,\Z)/\ik \cong\kZ$}}
   \put(195,0){\makebox(0,0){$\EE\cong\B(\Og,\R)/\Ik$}}
   \put(0,70){\makebox(0,0){$\UU$}}
   \put(90,70){\makebox(0,0){$\B(\Og,\Z)$}}
   \put(190,70){\makebox(0,0){$\B(\Og,\R)$}}
   \put(30,6){\makebox(0,0){$\sg$}}
   \put(143,6){\makebox(0,0){$\imath$}}
   \put(45,76){\makebox(0,0){$\psi$}}
   \put(138,76){\makebox(0,0){$\imath$}}
   \put(-8,35){\makebox(0,0){$\sk$}}
    \put(98,35){\makebox(0,0){$\pi_{\ik}$}}
     \put(198,35){\makebox(0,0){$J$}}
   
   \put(10,0){\vector(1,0){42}}
    \put(130,0){\vector(1,0){26}}
   \put(10,70){\vector(1,0){60}}
   \put(110,70){\vector(1,0){58}}
   \put(0,60){\vector(0,-1){50}}
   \put(90,60){\vector(0,-1){50}}
   \put(190,60){\vector(0,-1){50}}
\end{picture}
\end{center}
\bigskip

In particular
 $\nk=\sg\circ\sk$ is the  Euclidean numerosity function, and  
 the Euclidean field $\EE$ is  uniquely determined by $\sk$.
 \end{thm}

\noindent \begin{proof}\par 
Let $E$ be a set of ordinals: then $\sum_k\chi_E(k)$ is the Euclidean numerosity of $E$, by definition, and the corresponding counting function is precisely $\psi_E$. 
On the other hand, the transfinite sum of nonnegative integers has been defined above in a consistent way also inside the ring $\Ak$, so $\sk(E)=\sum_k\chi_E(k)$. 

We have already remarked  that the axiom \CA\ holds in $\Ak$ for transfinite sums of integers. Moreover the equality $\dsp$\ yields that also both axioms \DA\ and \PA\ hold. It follows that $\sk(E)=\sk(F)$  if and only if $\psi_E\equiv\psi_F\!\mod\ik$, and $\sg$ is uniquely and consistently defined by $\sg(\sk(E)-\sk(F))= [(\psi_E-\psi_F)\!\mod\ik$.

\par \qed \end{proof}


\begin{rem}
The last point of the above theorem shows that the Euclidean field
$\mathbb{E}$ arises quite naturally from a numerosity theory that
satisfies very reasonable assumptions which extend and interpret the 
Euclidean common notions.\end{rem}

 %

\section{The existence of the field of the Euclidean 
numbers}\label{constr}
In this section we ground on the particular kind of ordinals,
that we called  \emph{complete ordinals} in subsection \ref{fin}:
they will be used as convenient  ``check points'' for the values 
of the transfinite sums, and of the corresponding counting functions,
in order to produce a model of the field of the Euclidean numbers.

\subsection{The Euclidean ultrafilter $\mathcal{U}$}\label{ult}

Let 
$\eg=2^\eg$ be an  $\eg$-number in $\Og$, and let 
$\mathcal{F}$ be a filter on the set 
$\Og_{\eg}$ of the ordinals less than $\eg$. 
Given a function $G:\Og\to\F$ such that $G(\ell)\incl C(0)$ for all $\ell\in\Og$,
put
\[
\_G=
\{\ag=\sum_{\ell\qincl\cg_0} \eg^\ell\odot\cg_\ell\mid\, \forall\ell\qincl\cg_0\,( \cg_\ell\in G(\ell)) \,\} \ 
\]
so that
$$i=\eg^h\odot k\qincl\ag=\sum_{\ell\qincl{\cg_0}} \eg^\ell\odot\cg_\ell\in\_G\ \ \Iff\ \  \,h\qincl\cg_0\ \&\  k\qincl\cg_h$$

Clearly one has $\_G\cap \_H=\_K$, where $K(\ell)=G(\ell)\cap H(\ell)$, hence the family 
$\{\_F\mid F:\Og\to\F)\}$ is closed under finite intersection, so it generates a filter
$\_\F$ on $\Og$.
Similarly, the intersections $\_F\cap\Og_\eta$ generate a filter $\_\F_{|\eta}$ on $\Og_\eta$ for all $\eg$-numbers $\eta=2^\eta>\eg$.

\medskip

Since each ordinal  $\ag\in\Og$ is viewed in this context as the finite subset $L_\ag\incl\Og$, we call \emph{fine} a filter on $\Og_\eta$ if it is \emph{fine as a set of finite subsets} of $\Og_\eta$, \ie\ if it contains all intersections $C(\bg)\cap\P_\og(\Og_\eta)=\{\ag<2^\eta\mid\bg\qincl\ag\}$.
It is worth noticing that this identification works very well  for $\eg$-numbers $\eg=2^\eg$, because then $\ag<\eg$ if and only if $L_\ag\incl\Og_\eg$.

In Subsection \ref{fin}, we introduced the  filter $\QQ$ generated by the  family 
of the $\qincl$-cones $C(\bg)=
\{\ag\mid \bg\qincl\ag\,\}$  together with the families $D(\eta,h)$ of the
 $(\eta,h)$-complete ordinals
(see Lemma \ref{compl}), and  the corresponding filter  $\QQ_\eg$ restricted to $\Og_\eg$.

We  give now an inductive construction of a \emph{``Euclidean ultrafilter"} on $\Og$. 
\begin{thm}
\label{thm:figo}There exist a fine ultrafilter 
$\mathcal{U}$ 
on $\Omega$, and fine ultrafilters $\mathcal{U}(\eg)$ on $\Og_{{\eg}}$, for all $ \eg=2^\eg\in\Og$,  such that


\begin{itemize}
 
\item
${}\ \ \ \ \ \ \ \ 
\QQ\incl\mathcal{U}\ \ \ \
\ and \ \ \ \ \ 
\QQ_{\eg}
\incl\mathcal{U}(\eg);
$

 
\item ${}\ \ \ \ \ \ \ \ 
  \_{\U(\eg)}\incl\U\  \ \ and\ \ \ \_{\U(\eta)}_{|\eg}
  \incl\U(\eg)\ \ \   \ for\ all\ \ {\eta=2^\eta<\eg}.$


\end{itemize}
\end{thm}

\begin{proof}
    Let $\mathcal{U}(\omega)$ be a fine ultrafilter on $\Og_{\og}=\mathbb{N}$,  containing the set 
    
    $\{\ag<\og\mid \exists n\in\N.\, \ag=2^n-1\}$ (so $L_\ag=\{0,\ldots,n-1\})$.
    
Let $\la\eg_{j}\ra_{j\in\Og}$ be an enumeration of the $\eg$-numbers in $\Og$. We proceed to define ultrafilters $\U({\eg_j})$ on $\Og_{\eg_j}$ by induction on $j$.

Assume that ultrafilters $\mathcal{U}(\eg_{s})$ be conveniently defined for 
$s\le j$, and consider the filters
$\ 
\QQ_{\eg_{j+1}}  \mathrm{and}\ \
\_{\U({\eg_{s}})}_{|\eg_{j+1}}\, (s\le j)
$: 
each of them is obviously closed under finite intersections, so we need only to prove
 that $D\cap\_U \ne\0$ for all $D=D(E,H;\bg)\in\Ql_{\eg_{j+1}}$ and all $\_U$ corresponding to  $U:\Og_{\eg_{j+1}}\to\U(\eg_j)$.
 
Given $U:\Og_{\eg_{j+1}}\to\U(\eg_j)$ and $E,H\incl \Og_{\eg_{j+1}}$, $\bg<\eg_{j+1}$, we have to find  an ordinal $\ag=\sum_{\ell\qincl\cg_0}{\eg_{j}^\ell}\odot\cg_\ell>\bg$ with $\ell<\eg_{j+1}$ and $\cg_\ell\in U(\ell)$,
such that, for all $h\in H$, all $\eta\in E$, and all $k<\eta$
$$(2^ \eta)^h \odot k \qincl\ag=\sum_{\ell\qincl\cg_0}{\eg_{j}^\ell}\odot\cg_\ell\ \ \Iff\ \ k\qincl\ag=\sum_{\ell\qincl\cg_0}{\eg_{j}^\ell}\odot\cg_\ell\big.$$
Now distinguish four cases
\begin{enumerate}
  \item $\eta, h<\eg_j$, hence $(2^\eta)^h\odot k<\eg_j$, and then 
 $ (2^ \eta)^h \odot k \qincl\ag\ \ \Iff\ \ (2^ \eta)^h \odot k\qincl\cg_0;$
  \item $\eta=\eg_j$, hence $(2^ \eta)^h=\eg_j^h$, and then  
  $ (2^ \eta)^h \odot k \qincl\ag\ \ \Iff\ \ h\qincl \cg_0,\ k\qincl\cg_h;$ 
    \item $\eta<\eg_j,\ \eg_{j+1}> h\ge\eg_j$, hence $h=\eg_j\odot q_h\oplus r_h $, and $(2^\eta)^h=\eg_j^{q_h}\odot 2^{ \eta\odot r_h}$, and then 
 $ ~(2^ \eta)^h \odot k \qincl\ag\ \ \Iff\ \ q_h\qincl\cg_0,\ 2^{\eta\odot r_h}\odot k\qincl\cg_{q_h};$

   \item $\eg_{j+1}>\eta>\eg_j$, hence $2^\eta=\eg_j^{i}$ for some $i<\eg_{j+1}$, and then
   
   $~~~~~~~~~~~~~~ (2^ \eta)^h \odot k \qincl\ag\ \ \Iff\ \ i\odot h\qincl\cg_0,\   k\qincl\cg_{i\odot h}.$
\end{enumerate}
Let $E_1=E\cap\Og_{\eg_j}$, $H_1=H\cap\Og_{\eg_j}$, so $D(E_1,H_1)\in\QQ_{\eg_j}$.
Then one can choose:
\begin{enumerate}
  \item 
$\cg_0\in D(E_1,H_1)$ in order to meet case $1$.
Moreover one has to add finitely many elements to $L_{\cg_0}$ so as to have that $h, q_h,i\odot h\qincl\cg_0$ , in accord to cases $2,3,4$, for the respective values of  $h\in H$.

 Now one has to pick 
\item  $\cg_h=\cg_0$ for all $h\in H$, according to case $2$, since so $\ag\in D(\eg_j,H)=D(\eta,H)$; 
 \item 
$\cg_{q_h}=\cg_0$ for  $h=\eg_j\odot q_h\oplus r_h$, according to case $3$, since then\\
$~~2^{\eta\odot r_h}\odot k\qincl\cg_{q_h}=\cg_0\ \Imp\ k\qincl\ag$;
\item
$\cg_{i\odot h}=\cg_0$ for $2^\eta=\eg^i_j$, according to case $4$, since so $\ag\in D(\eta,h)=D(\eg_j,i\odot h)$. 
\end{enumerate} 
%
%
%
 So we can take $\mathcal{U}(\eta_{j+1})$ to be any ultrafilter
containing both families of sets. 

\medskip
For limit $j$, assume that the ultrafilters $\mathcal{U}(\eta_{s})$ have been conveniently 
defined for all  $s< j$. Then, for all $r<j$, all families of filters
\[
\_{\U({\eg_s})}\ \ \textrm{and}\ \QQ_{\eg_s},\ \ \ s<r
\]
are included in $\U(\eg_{r})$, by induction hypothesis, hence all of them share together the finite intersection property. Then also the union of all
 these families for $r<j$ has the \FIP , because $\eg<{\eg_j}$ implies
$\eg<{\eg_r}$ for some $r<j$. So any  ultrafilter 
containing this union can be taken to be 
$\mathcal{U}(\eg_{j})$.


\smallskip

Finally, after having  defined the ultrafilter $\U(\eg_{j})$ for all $j\in\Omega$, pick as
$\mathcal{U}$  an ultrafilter on $\Og$ containing the filters\,
$
\_{\U({\eg_j})} \
\mathrm{and}\ 
\QQ_{{\eg_{j}}},
$
for all $j\in\Og$.  The finite intersection property follows by the same argument 
 of the 
 case of limit $j$, 
and all 
conditions are fulfilled. 

 \qed
\end{proof}
\bigskip

We name \emph{Euclidean} an ultrafilter satisfying the conditions of
$\mathcal{U}(\Omega)$, since any such ultrafilter
provides a model of the field of the Euclidean numbers,
as we shaw in the next subsection.


\subsection{A construction of the field of the Euclidean 
numbers}\label{cons}
We may now present a construction of a field enjoying the properties 
of the Euclidean numbers 
axiomatized in Section 2.
\begin{itemize}
  \item Let{ $\mathbb{R}^{\Og}$} \foreignlanguage{english}
be {the algebra
of all real valued functions on $\Og$, and let
$$\mathscr{B}(\Og,\,\mathbb{R})= \{\varphi_\xb\in\mathbb{R}^{\Og}
\mid\,\xb\in\S(\Og,\R))\}
$$
be the subalgebra of the counting functions.
}Recall that, by Theorem \ref{S=B}, every element of $\B(\Og,\R)$ is $j$-periodic for some $j\in\Og$, and conversely any  $j$-periodic $\psi\in\R^\Og$ agrees with one and only one $\fg_\xb\in\B(\Og,\R)$.

  \item Let $\mathcal{U}$ be a Euclidean ultrafilter on $\Og$,  
let $\Ik$ be the corresponding maximal ideal of 
$\mathscr{B}(\Og,\,\mathbb{R})$
\[
\mathfrak{I}:=\left\{ \varphi\in\mathscr{B}(\Og,\,\mathbb{R})\,|\,
\exists Q\in\mathcal{U}\,\forall\bg\in Q\, .\,\varphi(\bg)=0\right\} 
\]
 and  let 
\[ J\,:\,\mathscr{B}(\Og,\,\mathbb{R})
\longrightarrow\,
\mathscr{B}(\Og,\,\mathbb{R})/\mathfrak{I}=\mathbb{E}
\]
be the canonical homomorphism onto the corresponding quotient field, 
which becomes an ordered field with the ordering induced by the 
natural partial ordering of $\mathscr{B}(\Og,\,\mathbb{R})$.

  \item  Denote by $\left[\varphi\right]$ the coset $J(\varphi)=\varphi+\Ik\in\EE$
of the function 
$\varphi\in\mathscr{B}(\Og,\,\mathbb{R})$, and define the sum 
$\sum_{k}x_{k}=\Sg(\xb)$ of the $\Og$-sequence $\xb\in \S(\Og,\R)$ as the coset $[\fg_{x}]=J(\fg_{\xb})$ of the 
corresponding counting 
function $\fg_{\xb}\in\B(\Og,\R)$.

\end{itemize}

\medskip
Then we have 
\begin{thm}\label{mod}
The ordered field $\EE$ satisfies the axiom \RA, and the transfinite sum
$\Sg:\S(\Og,\R)\to\EE$  satisfies the axioms
\CA, \PA\ and \DA; moreover $\Sg$
can be uniquely extended to a $\R$-linear 
map $\Sg:\S(\Og,\EE)\to\EE$ so as to satisfy the axioms \LA\ and  \SA.
Hence $\EE$ becomes a Euclidean field.
\end{thm}
\begin{proof}
 \begin{itemize}   
 \item   The  axiom \RA\ holds by definition.
     \item  
 The axiom \CA\ holds because the ultrafilter $\U$ contains 
    all cones $C(\bg)$.

    \item
    The  map  $J$ is an $\R$-algebra homomorphism, so 
$$(\fg_{\xb}\fg_{\yb})(\ag)=\fg_{\xb}(\ag)\fg_{\yb}(\ag)=
\sum_{h\qincl\ag}x_{h}\sum_{k\qincl\ag}y_{k}=
\sum_{h,k\qincl\ag}x_{h}y_{k}
$$ 
Hence, putting  $z_{i}=
\sum_{h,k\qincl i}x_{h}y_k$, one has $\sum_{i\qincl\ag}z_i=\sum_{h,k\qincl\ag}x_{h}y_{k}$, hence
$$(\sum_{h}x_{h})(\sum_{k}y_{k})=[\fg_{\xb}][\fg_{\yb}]=[\fg_\zb]=\sum_{i}z_{i}=\sum_{h,k} x_h y_k,$$
and the axiom \PA\ holds.

\item The axiom \DA\ holds because the ultrafilter $\U$ contains 
    all sets $D(\eta,h)$, and one has, for $\ag\in D(2^\eta,1)$ and $h,k<\eta$:
    $$h,k\qincl\ag\ \ \Iff\ \ 2^\eta\odot h+k\qincl\ag,\ \ \textrm{hence}\ \ 
    \sum_{h,k\qincl\ag} x_{hk} =\sum_{2^\eta\odot h+k\qincl\ag} x_{hk}.$$

   \item  Now we extend the map $\Sg$ to $\S(\Og,\EE)$ so as to satisfy the 
    axiom \SA.  \\
    Let $\xi_{h}=\sum_{k}x_{hk}$ be given, 
    assume that 
  $x_{hk}=0$ for $h,k\ge \eg,$  
and put

$\sum_{h}\xi_{h}=\sum_{h}\sum_{k} x_{hk}=\sum_{i}y_{i}\ \ \mathrm{with}\ \  y_i $ defined according to \SA. 

Then the axiom \LA\ holds, because  
    the map $\yb\mapsto \fg_{\yb}$ is $\R$-linear, so also the extended definition of $\Sg$ is $\R$-linear, provided it is well defined, \ie\  independent of the representation of $\xi_h$ as a transfinite sum of reals.


\item  Now let 
$\xi_{h}=\sum_{k}x'_{hk}$ be alternative expressions of the same elements  
$\xi_{h}\in\EE$, assume w.l.o.g. that also  $x'_{hk}=0$ for $h,k\ge \eg,$ and define $y'_{i}$ accordingly. Then, for each 
$h<\eg$ there exists a qualified set $U_{h}\in\U(\eg)$
such that \ \
$\sum_{k\qincl\cg}x_{hk}=\sum_{k\qincl\cg}x'_{hk}$ \ \ for all 
$\cg\in U_{h}$.

Assume \wlg\ that $U_h\incl C(0)$ for all $h\in\Og$, and consider the function $U:h\mapsto U_h\in\U(\eg)$:     the ultrafilter $\U$ 
     contains the filter 
    $ \_{\U({\eg})}$, \ie\ all  sets 
    $$\_U=\{\sum_{h\qincl\cg_0}\eg^h\odot\cg_h\mid\,\forall h\qincl\cg_0\ \cg_h\in U_h\}
$$
So,
for $\ag=\sum_{h\qincl\cg_0}\eg^h\odot\cg_h\in \_{U}$, and $h,k<\eg$, one has $0\qincl\cg_0\in U_0$, and
$$\eg^h\odot k\qincl\ag \ \ \Iff\ \ (h=0\ \&\ k\qincl\cg_0)\ \vee\ (0<h\qincl\cg_0\ \&\ k\qincl\cg_h),$$
whence
  $$\sum_{i\qincl\ag}y_{i}=
\sum_{h\qincl\cg_0}  \sum_{k\qincl\cg_h} y_{\eg^h\odot k}= \sum_{k\qincl\cg_0}x_{0k}
+\sum_{0\ne h\qincl\cg_0}\big( x_{h0}+x_{hh}+ \sum_{h\ne k\qincl\cg_h} x_{hk}\big)
$$
$$= \sum_{k\qincl\cg_0}x'_{0k}
+\sum_{0\ne h\qincl\cg_0}\big( x'_{h0}+x'_{hh}+ \sum_{h\ne k\qincl\cg_h} x'_{hk}\big)=
 \sum_{h\qincl\cg_0}  \sum_{k\qincl\cg_h} y'_{\eg^h\odot k} =\sum_{i\qincl \ag}y'_{i},
$$  
 and so $\sum_{i}y_{i}=\sum_{i}y'_{i}.$

\end{itemize}  
\qed
\end{proof}

\medskip  

So the existence of the Euclidean field $\EE$ is implied by that 
of Euclidean ultrafilters. 


\section*{Appendix: \emph{The Euclidean continuum}}\label{app}\ \

In ancient geometry, \emph{lines} and \emph{segments} are
not considered as \emph{sets of points}; on the contrary, 
in the last two centuries, the
reductionistic attitude of modern mathematics has tried and described
the Euclidean geometry through a \emph{set theoretic interpretation}. 
So the\emph{ Euclidean continuum} has
been identified with the \emph{Dedekind continuum} and the Euclidean line
has been identified with (or at least considered \emph{isomorphic} to)
the \emph{set of real numbers} (once
an origin $O$ and a unit segment $OA$ have been fixed). Although this
identification be almost universally accepted today, nevertheless it contradicts
various theorems of the Euclidean geometry.

As an important example, we cite the Euclidean 
statement that a segment $AB$ can be \emph{divided in two congruent segments}
$AM$ and $MB$. If $AB$ is identified with the Dedekind continuum
then, either $AM$ has a maximum or $MB$ has a minumum. Thus
$AM$ and $MB$ are not congruent, hence \emph{stricto sensu} the Dedekind continuum
is not a correct model for the Euclidean continuum. In order to construct
a consistent model, we are forced to assume that the \emph{points} $A,B$
and $M$ \emph{do not belong} to the segment $AB$. So the picture which
comes out of the\emph{ Euclidean  straight line} is a \emph{linearly ordered} set
$\mathfrak{E}$, where the Euclidean segment $AB$ is a \emph{subset} of $\mathfrak{E}$
which \emph{cannot} be identified with the set theoretical segment 
$$S(A,B):=\left\{ X\in\mathfrak{E}\ |\ A<X<B\right\}, \ \mathrm{because}\
M\in S(A,B)\setminus AB.
$$
So we might better think of the segment $AB$ as a set of \emph{atoms} with 
lots of \emph{empty
spaces} between them.

In  the Euclidean theory of proportions, a \emph{set}
of magnitudes can be put in biunique correspondence with the \emph{lengths}
of the segments, and the lengths of segments satisfy the \emph{axiom of Archimedes}.
Hence, assuming also this axiom in
  the set theoretic interpretation of (oriented) segments, they build
a field isomorphic to the real field, so that, after a suitable 
identification, 
\begin{equation}
\mathbb{R}\subset\mathfrak{E}\label{uno}
\end{equation}
and, since $AB\subseteqq S(A,B)\setminus\mathbb{R}$, we can assume
that \ $
AB=S(A,B)\setminus\mathbb{R}.
$

Now take an atom $\mathsf{s}$ in $AB$: the \textit{distance}
between $A$ and $\mathsf{s}$ cannot be \emph{measured }by a 
\emph{length}, because
 $S(A,\mathsf{s})$ is not a segment and no segment is congruent
to $S(A,\mathsf{s})$. So  the\emph{ lengths}
cannot be as many as the \emph{distances}; if the lengths are
in biunique correspondence with $\mathbb{R}$, the distances can
be put in biunique correspondence with the full $\mathfrak{E}$.
Assuming that also the distances form a field, the inclusion (\ref{uno})
implies that $\mathfrak{E}$ is  a \emph{non-Archimedean} field. So the Euclidean
continuum leads to the non-Archimedean geometry as described by Veronese
at the end of the XIX century (see \cite{veronese,veronese2})\footnote{
~Actually Veronese considered also infinitesimal segments, but this
is matter of definitions that will be not discussed here.} 
However there are may non-Archimedean fields which contain the real
numbers. Moreover, every non-Archimedean field has gaps. Thus the 
question arises as to wether a non-Archimedean field provides a more satisfactory
model
of the Euclidean continuum than the Dedekind
continuum. 

In a naive way, a \emph{continuum} is a linearly ordered set without \emph{holes}.
In contrast with our intuition, a set $X$ which satisfies the following
property 
\begin{equation}
\forall a,b\in X,\ a<b,\ \exists c\in X,\ a<c<b\label{eq:13}
\end{equation}
is not a continuum since there are \emph{holes} (think \pes\ to a segment
of rational numbers: here the irrational numbers can be considered
holes). However also the Dedekind continuum is not satsfactory:  
the arguments outlined above yield that
the lenghts form a Dedekind continuum, but there are distances which
are not lenghts. Thus in a sense also the Dedekind continuum contains
holes, represented by the distances which are not lenghts. So we are
tempted to give the following definition which generalizes (\ref{eq:13}): 
\begin{quote}
\emph{a linearly ordered ``set" $X$ is an continuum if given two subsets
$A$ and $B$ such that $\forall a\in A,\ \forall b\in B,\ a<b,$
then} 
\begin{equation}
\exists c\in X,\ a<c<b\label{eq:14}
\end{equation}
\end{quote}
Assuming this definition, a continuum is a \emph{proper class} in the sense of 
Von Neumann-Bernays-G\"odel (NBG) class theory. This is the point of view
of Ehrlich \cite{el12}, where the \emph{class of surreal numbers} is
viewed as the \emph{absolute arithmetic continuum}. In fact, as far as the order
structure is concerned, the surreal numbers have the order type described
by (\ref{eq:14}).

We prefer to have the continuum to be a \emph{set}; so we assume the
existence of an inaccessible ordinal and we give the following definition:
\begin{defn}
\emph{A linearly ordered set $X$ is a Euclidean continuum if it is
$\Omega$-saturated, i.e. given any two subsets $A$ and $B$ such
that $|A|,|B|<|X|=|\Omega|$ and $\forall a\in A,\ \forall b\in B,\ a<b,$
we have that} 
\begin{equation}
\exists c\in X,\ a<c<b\label{eq:14-1}
\end{equation}
\end{defn}

Grounding on this discussion, we are led to assume that the Euclidean line
$\mathfrak{E}$ is an \emph{Euclidean continuum} equipped with the structure
of  \emph{real closed} field, so it is \emph{isomorphic} to the field of Euclidean
numbers $\mathbb{E}$. 

The fact that $\mathfrak{E}$ has to be real
closed is a reasonable request: it has to contain not only all 
Euclidean numbers \emph{stricto sensu},\footnote{~
\Ie\ those constructible by \emph{ruler and compass}.} but also all 
zeroes of sign-changing polynomials.

\smallskip
At this point it seems appropriate to explain why, in our opinion, the
order type of $\mathfrak{E}$ must be $\eta_{\Omega}$ rather that
$\eta_{\alpha}$\footnote{~$\eta_{\alpha}$ is the order type of an
$\alpha$-saturated ordered
set of size $\ag$; assuming the generalized continuum hypotesis, such
sets exist for every regular $\ag$. 
} for some $\alpha<\Omega$.

We sketch two arguments. The first is \emph{set theoretical}. The condition
(\ref{eq:14}) seems implicit in the notion of (absolute) continuum,
but working with proper classes meets several technical limitations, 
which do not arise in workig with inaccessible ordinals; so
it seems appropriate to ``truncate'' the universe at the first 
inaccessible level. In doing this, the Euclidean continuum becomes 
indiscernible from the absolute continuum of Ehlrich, because then 
the class of surreal numbers would
be a field isomorphic to $\mathbb{E}$.

Moreover, taking into account the role of the Euclidean numbers as \emph{numerosities}, the use of a saturated real closed field $\FF$ of accessible cardinality seems inappropriate: it contradicts  the natural assumption that a  \emph{powerset}, or a set of f\emph{unctions}, has a numerosity in $\FF$ whenever the original \emph{set} (the \emph{domain}) has a numerosity in $\FF$.

The second argument is \emph{``geometric"}. The assumption  $|\mathfrak{E}|=\Omega$ 
yields $\mathfrak{E}\cong\mathbb{E}$, and hence it inherits a very rich
structure. In particular, every Euclidean number being a transfinite sum
of reals, one has that any \emph{distance}
is a transfinite (algebraic) sum of \emph{lengths}.
For instance
\begin{equation}
\sum_{k\in[0,\omega)}\,\frac{1}{2^{k}}=2-\frac{1}{2^{\omega-1}}\label{eq:martina}
\end{equation}
(see equation (\ref{eq:marta}) in subsection \ref{hyperf}) so adding infinitely many segments of lenght $2^{-k}$
cannot provide a segment of length $2$, but  only a quantity (a \emph{distance})
which is (infinitesimally)  smaller than $2$. In the Euclidean world, Achilles
will never properly reach the turtle, remaining forever at infinitesimal
distance (these issues have been discussed also in \cite{BF16,BF19}). Assuming that in the physical world infinitesimal
quantities cannot be measured and  so do not count, the turtle is reached. 
In our opinion,
this joke might enphasize the fact that non-Archimedean fields
(and, in the future, hopefully also the Euclidean field) might be very useful
in building models of natural phenomena (see some examples in
\cite{A,ultra,BHW13,belu2012,BHW16}).

\bigskip{}

\bibliographystyle{asl}

\end{document}